\documentclass[11pt,oneside,english]{amsart}
\usepackage[T1]{fontenc}
\usepackage[latin9]{inputenc}
\usepackage{geometry}
\usepackage{babel}
\usepackage{float}
\usepackage{amsthm}
\usepackage{amssymb}
\usepackage{stmaryrd}
\usepackage[unicode=true,pdfusetitle,
 bookmarks=true,bookmarksnumbered=false,bookmarksopen=false,
 breaklinks=false,pdfborder={0 0 1},backref=false,colorlinks=false]
 {hyperref}

\makeatletter

\providecommand{\tabularnewline}{\\}

\numberwithin{equation}{section}
\theoremstyle{plain}
\newtheorem{thm}{\protect\theoremname}[section]
  \theoremstyle{definition}
  \newtheorem{defn}[thm]{\protect\definitionname}
  \theoremstyle{remark}
  \newtheorem{rem}[thm]{\protect\remarkname}
  \theoremstyle{definition}
  \newtheorem{example}[thm]{\protect\examplename}
  \theoremstyle{plain}
  \newtheorem{prop}[thm]{\protect\propositionname}
  \theoremstyle{plain}
  \newtheorem{lem}[thm]{\protect\lemmaname}
  \theoremstyle{plain}
  \newtheorem{cor}[thm]{\protect\corollaryname}

\usepackage{tikz,diagrams,mathtools}
\usepackage{palatino, euscript}

\newcommand{\arc}{\,\tikz[baseline=\the\dimexpr\fontdimen22\textfont2\relax ]{\draw[->] plot [smooth] coordinates {(0,0)(0.05,0.2)(0,0.4)}; }\, }

\newcommand{\arcd}{\,\tikz[baseline=\the\dimexpr\fontdimen22\textfont2\relax ]{\draw[<-] plot [smooth] coordinates {(0,0)(0.05,0.2)(0,0.4)}; }\, }

\newcommand{\arcs}{\,\tikz[baseline=\the\dimexpr\fontdimen22\textfont2\relax ]{\draw[->] plot [smooth] coordinates {(0,0)(0.05,0.2)(0,0.4)};\draw[->] plot [smooth] coordinates {(0.4,0)(0.35,0.2)(0.4,0.4)}; }\, }

\newcommand{\arcsd}{\,\tikz[baseline=\the\dimexpr\fontdimen22\textfont2\relax ]{\draw[<-] plot [smooth] coordinates {(0,0)(0.05,0.2)(0,0.4)};\draw[<-] plot [smooth] coordinates {(0.4,0)(0.35,0.2)(0.4,0.4)}; }\, }

\newcommand{\arcsud}{\,\tikz[baseline=\the\dimexpr\fontdimen22\textfont2\relax ]{\draw[->] plot [smooth] coordinates {(0,0)(0.05,0.2)(0,0.4)};\draw[<-] plot [smooth] coordinates {(0.4,0)(0.35,0.2)(0.4,0.4)}; }\, }

\newcommand{\arcsacross}{\,\tikz[baseline=\the\dimexpr\fontdimen22\textfont2\relax ]{\draw[->] plot [smooth] coordinates {(0,0)(0.2,0.05)(0.4,0)};\draw[->] plot [smooth] coordinates {(0.4,0.4)(0.2,0.35)(0,0.4)}; }\, }

\newcommand{\wideedge}{\,\tikz[baseline=\the\dimexpr\fontdimen22\textfont2\relax ]
{\draw[->] (0,0)--(0.4,0.4); 
\draw[->] (0.4,0)--(0,0.4); 
\draw[fill=black] (0.2,0.2) circle (0.075cm);
}\,}

\newcommand{\wideedged}{\,\tikz[baseline=\the\dimexpr\fontdimen22\textfont2\relax ]
{\draw[<-] (0,0)--(0.4,0.4); 
\draw[<-] (0.4,0)--(0,0.4); 
\draw[fill=black] (0.2,0.2) circle (0.075cm);
}\,}

\newcommand{\vir}{\,\tikz[baseline=\the\dimexpr\fontdimen22\textfont2\relax ]{\draw[->] (0,0)--(0.4,0.4); \draw[->] (0.4,0)--(0,0.4); \draw (0.2,0.2) circle (0.075cm);}\,}

\newcommand{\vird}{\,\tikz[baseline=\the\dimexpr\fontdimen22\textfont2\relax ]{\draw[<-] (0,0)--(0.4,0.4); \draw[<-] (0.4,0)--(0,0.4); \draw (0.2,0.2) circle (0.075cm);}\,}

\newcommand{\posx}{\,\tikz[baseline=\the\dimexpr\fontdimen22\textfont2\relax]{
\draw[->] (0.4,0)--(0,0.4);
\draw[ultra thick, white] (0,0)--(0.4,0.4);
\draw[ultra thick, white] (0.05,0)--(0.45,0.4);
\draw[ultra thick, white] (-0.05,0)--(0.35,0.4);
\draw[->] (0,0)--(0.4,0.4);
}\,}

\newcommand{\negx}{\,\tikz[baseline=\the\dimexpr\fontdimen22\textfont2\relax]{
\draw[->] (0,0)--(0.4,0.4);
\draw[ultra thick, white] (0.4,0)--(0,0.4);
\draw[ultra thick, white] (0.4,0.05)--(0.05,0.4);
\draw[ultra thick, white] (0.4,0-0.05)--(-0.05,0.4);
\draw[->] (0.4,0)--(0,0.4);
}\,}

\newcommand{\arcarc}{\,\tikz[baseline=\the\dimexpr\fontdimen22\textfont2\relax]{
\draw[->] (0,0)--(0,0.4);
\draw (0.4,0.2) circle (.2cm);
\draw[->] (0.39,0.4)--(0.41,0.4);
\draw[<-] (0.8,0)--(0.8,0.4);
}\,}

\newcommand{\circleO}{\,\tikz[baseline=\the\dimexpr\fontdimen22\textfont2\relax]{
\draw (0.4,0.2) circle (.2cm);
\draw[->] (0.39,0.4)--(0.41,0.4);
}\,}

\newcommand{\arcvir}{\,\tikz[baseline=\the\dimexpr\fontdimen22\textfont2\relax]{
\draw[->] (0,0)--(0,0.4);
\draw plot [smooth] coordinates {(0.4,0.4)(0.3,0.2)(0.4,0)};
\draw[<-] (0.4,0)--(0.8,0.4); 
\draw[<-] (0.8,0)--(0.4,0.4); 
\draw (0.6,0.2) circle (0.075cm);
}\,}

\newcommand{\virarc}{\,\tikz[baseline=\the\dimexpr\fontdimen22\textfont2\relax]{
\draw[<-] (0.8,0)--(0.8,0.4);
\draw plot [smooth] coordinates {(0.4,0.4)(0.5,0.2)(0.4,0)};
\draw[->] (0.0,0)--(0.4,0.4); 
\draw[->] (0.4,0)--(0.0,0.4); 
\draw (0.2,0.2) circle (0.075cm);
}\,}

\newcommand{\virvir}{\,\tikz[baseline=\the\dimexpr\fontdimen22\textfont2\relax]{
\draw (0.0,0)--(0.4,0.4); 
\draw[->] (0.4,0)--(0.0,0.4); 
\draw (0.2,0.2) circle (0.075cm);
\draw (0.4,0)--(0.8,0.4); 
\draw[<-] (0.8,0)--(0.4,0.4); 
\draw (0.6,0.2) circle (0.075cm);
}\,}

\newcommand{\arcedge}{\,\tikz[baseline=\the\dimexpr\fontdimen22\textfont2\relax]{
\draw[->] (0,0)--(0,0.4);
\draw plot [smooth] coordinates {(0.4,0.4)(0.3,0.2)(0.4,0)};
\draw[<-] (0.4,0)--(0.8,0.4); 
\draw[<-] (0.8,0)--(0.4,0.4); 
\draw[fill=black] (0.6,0.2) circle (0.075cm);
}\,}

\newcommand{\edgearc}{\,\tikz[baseline=\the\dimexpr\fontdimen22\textfont2\relax]{
\draw[<-] (0.8,0)--(0.8,0.4);
\draw plot [smooth] coordinates {(0.4,0.4)(0.5,0.2)(0.4,0)};
\draw[->] (0.0,0)--(0.4,0.4); 
\draw[->] (0.4,0)--(0.0,0.4); 
\draw[fill=black] (0.2,0.2) circle (0.075cm);
}\,}

\newcommand{\edgeedge}{\,\tikz[baseline=\the\dimexpr\fontdimen22\textfont2\relax]{
\draw[->] (0.0,0)--(0.4,0.4); 
\draw[->] (0.4,0)--(0.0,0.4); 
\draw[fill=black] (0.2,0.2) circle (0.075cm);
\draw[<-] (0.4,0)--(0.8,0.4); 
\draw[<-] (0.8,0)--(0.4,0.4); 
\draw[fill=black] (0.6,0.2) circle (0.075cm);
}\,}

\newcommand{\viredge}{\,\tikz[baseline=\the\dimexpr\fontdimen22\textfont2\relax]{
\draw[->] (0.0,0)--(0.4,0.4); 
\draw[->] (0.4,0)--(0.0,0.4); 
\draw (0.2,0.2) circle (0.075cm);
\draw[<-] (0.4,0)--(0.8,0.4); 
\draw[<-] (0.8,0)--(0.4,0.4); 
\draw[fill=black] (0.6,0.2) circle (0.075cm);
}\,}

\newcommand{\virnegxd}{\,\tikz[baseline=\the\dimexpr\fontdimen22\textfont2\relax]{
\draw (0.0,0)--(0.4,0.4); 
\draw[->] (0.4,0)--(0.0,0.4); 
\draw (0.2,0.2) circle (0.075cm);
\draw(0.4,0)--(0.8,0.4);
\draw[ultra thick, white] (0.8,0)--(0.5,0.3);
\draw[<-] (0.8,0)--(0.4,0.4);
}\,}

\newcommand{\posxnegxd}{\,\tikz[baseline=\the\dimexpr\fontdimen22\textfont2\relax]{
\draw[->] (0.4,0)--(0,0.4);
\draw[ultra thick, white] (0,0)--(0.4,0.4);
\draw[ultra thick, white] (0.05,0)--(0.45,0.4);
\draw[ultra thick, white] (-0.05,0)--(0.35,0.4);
\draw (0,0)--(0.4,0.4);
\draw (0.4,0)--(0.8,0.4);
\draw[ultra thick, white] (0.8,0)--(0.5,0.3);
\draw[<-] (0.8,0)--(0.4,0.4);
}\,}

\newcommand{\negxposxd}{\,\tikz[baseline=\the\dimexpr\fontdimen22\textfont2\relax]{
\draw[<-] (0.8,0)--(0.4,0.4);
\draw[ultra thick, white] (0.4,0)--(0.8,0.4);
\draw[ultra thick, white] (0.45,0)--(0.85,0.4);
\draw[ultra thick, white] (0.35,0)--(0.75,0.4);
\draw(0.4,0)--(0.8,0.4);
\draw (0,0)--(0.4,0.4);
\draw[ultra thick, white] (0.4,0)--(0.1,0.3);
\draw[->] (0.4,0)--(0.0,0.4);
}\,}

\newcommand{\xx}{\mathbf{x}}
\newcommand{\yy}{\mathbf{y}}

\newcommand{\zz}{\mathbf{z}}
\newcommand{\pp}{\mathbf{p}}
\newcommand{\PP}{\mathcal{P}}

\newcommand{\QQ}{\mathbb{Q}}
\newcommand{\ZZ}{\mathbb{Z}}
\newcommand{\HH}{\mathcal{H}}
\newcommand{\Cone}{\text{Cone}}

\renewcommand{\G}{\Gamma}

\newcommand{\Br}{\text{Br}}

\newarrow {DotsLine}{.}{.}{.}{.}{.} 

\makeatother

  \providecommand{\corollaryname}{Corollary}
  \providecommand{\definitionname}{Definition}
  \providecommand{\examplename}{Example}
  \providecommand{\lemmaname}{Lemma}
  \providecommand{\propositionname}{Proposition}
  \providecommand{\remarkname}{Remark}
\providecommand{\theoremname}{Theorem}

\begin{document}

\title{HOMFLY-PT homology for general link diagrams and braidlike isotopy}

\author{Michael Abel}

\date{\today}

\email{maabel@math.duke.edu}

\address{Department of Mathematics, Duke University, Durham, NC 27708 }
\begin{abstract}
Khovanov and Rozansky's categorification of the HOMFLY-PT polynomial
is invariant under braidlike isotopies for any link diagram and Markov
moves for braid closures. To define HOMFLY-PT homology, they required
a link to be presented as a braid closure, because they did not prove
invariance under the other oriented Reidemeister moves. In this text
we prove that the Reidemeister IIb move fails in HOMFLY-PT homology
by using virtual crossing filtrations of the author and Rozansky.
The decategorification of HOMFLY-PT homology for general link diagrams
gives a deformed version of the HOMFLY-PT polynomial, $P^{b}(D)$,
which can be used to detect nonbraidlike isotopies. Finally, we will
use $P^{b}(D)$ to prove that HOMFLY-PT homology is not an invariant
of virtual links, even when virtual links are presented as virtual
braid closures.
\end{abstract}

\maketitle
\tableofcontents

\section{Introduction}

Khovanov and Rozansky in \cite{KR08b} introduced a triply-graded
link homology theory categorifying the HOMFLY-PT polynomial. The construction
given in \cite{KR08b} of Khovanov-Rozansky HOMFLY-PT homology, or
briefly HOMFLY-PT homology, is an invariant of link diagrams up to
braidlike isotopy (isotopies which locally resemble isotopies of a
braid) and Markov moves for closed link diagrams. However, Khovanov
and Rozansky were not able to prove invariance under all oriented
Reidemeiester moves. In particular, they could not prove the Reidemeister
IIb move, and in fact expected that it would fail in general. Because
of this, they required that a link be presented as a braid closure
so that HOMFLY-PT homology would be an invariant of links. In this
text we will directly address this issue by proving the failure of
the Reidemeister IIb move in HOMFLY-PT homology, and explore the consequences
of this failure.

The framework of HOMFLY-PT homology can be extended to include the
use of ``virtual crossings'', degree 4 vertices which are not actually
positive or negative crossings. The author and Rozansky in \cite{AbRoz14}
proved that a filration can be placed on the chain complex whose homology
is HOMFLY-PT homology. The associated graded complex of this filtration
is described using diagrams containing only virtual crossings. The
filtration allows us to rewrite the chain complexes in an illuminating
manner, allowing us to see new isomorphisms which would be difficult
to see otherwise. Using the framework of virtual crossing filtrations
we prove the following theorem.
\begin{thm}
[See Theorem \ref{thm:CatRIIb}] Suppose $D_{1}$, $D_{2}$, and
$D_{3}$ are oriented link diagrams which are identical except in
the neighborhood of a single point. Suppose in the neighborhood of
that single point, $D_{1}$ is $\posxnegxd$, $D_{2}$ is $\arcsacross$,
and $D_{3}$ is $\virnegxd$. $\HH(D_{1})\simeq\HH(D_{3})$ up to
a grading shift, while $\HH(D_{1})\not\simeq\HH(D_{2})$ in general.
\end{thm}
In Section \ref{sec:GenDia} we prove the above theorem and give an
explicit example of a diagram of the unknot that does not have the
HOMFLY-PT homology of the unknot (Example \ref{exa:RIIbFail}). Recall
$\HH(D)$ is a triply-graded vector space. Suppose $d_{ijk}=\dim(\HH(D)_{i.j.k})$,
then we can define the Poincaré series of $\HH(D)$ as 
\begin{equation}
\PP(D)=\sum_{i,j,k\in\ZZ}d_{ijk}q^{i}a^{j}t^{k}.\label{eq:PoincareSeries}
\end{equation}

Let $P(D)$ denote the HOMFLY-PT polynomial of the link diagram $D.$
Murakami, Ohtsuki, and Yamada introduced a state sum formulation of
the HOMFLY-PT polynomial commonly called the MOY construction \cite{MOY}.
Their approach resolves a link diagram into a $\ZZ(q,a)$-linear combination
of oriented planar $4$-regular graphs. They give relations which
evaluate each such planar graph as an element of $\ZZ(q,a).$ The
resulting rational function from this process for any link diagram
$D$ is its HOMFLY-PT polynomial $P(D).$

We now define a deformed HOMFLY-PT polynomial $P^{b}(D)=\left.\PP(D)\right|_{t=-1}$.
In the case that $D$ is presented as a braid closure then $P^{b}(D)=P(D)$.
However, this is not true for general link diagrams. We collect properties
of $P^{b}(D)$ into the following theorem.
\begin{thm}
[See Theorem \ref{thm:PBPoly}] Let $D$ be a link diagram. $P^{b}(D)$
is an invariant of link diagrams up to braidlike isotopy satisfying
the skein relation $qP^{b}(\posx)-q^{-1}P^{b}(\negx)=(q-q^{-1})P^{b}(\arcs).$
Furthermore, $P^{b}(D)$ satisfies the relations in Figure \ref{fig:IntroPB}.
\end{thm}
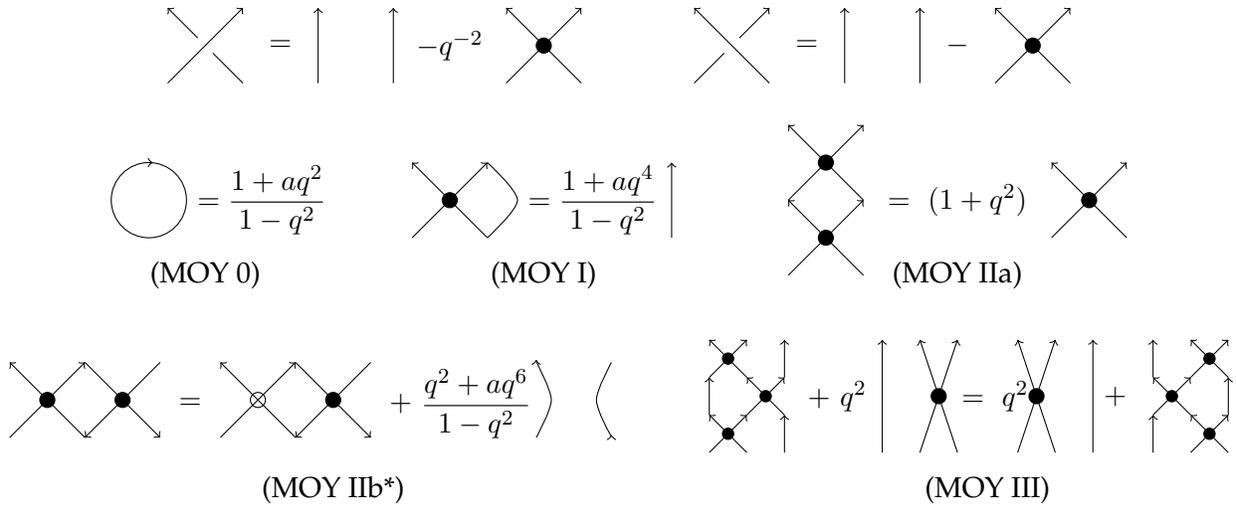
\begin{figure}[h]
\begin{tikzpicture} \draw[->] (3,0)--(4,1);\draw(4,0)--(3.6,0.4);\draw[->] (3.4,0.6)--(3,1); \node at (4.5,0.5) {$=$}; \draw[->] (5,0)--(5,1); \draw [->] (6,0)--(6,1);  \node at (6.75,0.5){$-q^{-2}$}; \draw[->] (7.5,0)--(8.5,1);\draw[->] (8.5,0)--(7.5,1);  \draw[fill=black] (8,0.5) circle (0.1cm);
\draw[->] (11,0)--(10,1);\draw(10,0)--(10.4,0.4);\draw[->] (10.6,0.6)--(11,1); \node at (11.5,0.5) {$=$}; \draw[->] (12,0)--(12,1); \draw [->] (13,0)--(13,1);  \node at (13.5,0.5){$-$}; \draw[->] (14,0)--(15,1);\draw[->] (15,0)--(14,1);  \draw[fill=black] (14.5,0.5) circle (0.1cm);
\end{tikzpicture}

\vspace{.2in}

\begin{tikzpicture} \draw (-3.5,1) circle (.5cm); \draw[->] (-3.5,1.5)--(-3.45,1.5); \node at (-2,1) {$=\dfrac{1+aq^2}{1-q^2}$}; \node at (-2.75,-0) {(MOY 0)};
\draw[->](0,0.5)--(1,1.5); \draw[->](1,0.5)--(0,1.5); \draw[fill=black](0.5,1) circle (0.1cm); \draw plot [smooth] coordinates {(1,1.5)(1.4,1)(1,0.5)}; \node at (2.4,1) {$= \dfrac{1+aq^4}{1-q^2}$}; \draw[->] (3.45,0.5)--(3.45,1.5);
\node at (1.75,-0) {(MOY I)}; \draw[->](5,0)--(6,1); \draw[->](6,0)--(5,1); \draw[fill=black](5.5,0.5) circle (0.1cm); \draw[->](5,1)--(6,2); \draw[->](6,1)--(5,2); \draw[fill=black](5.5,1.5) circle (0.1cm); \node at (7.25,1) {$=\,\,(1+q^2)$}; \draw[->](8.5,0.5)--(9.5,1.5); \draw[->](9.5,0.5)--(8.5,1.5); \draw[fill=black](9,1) circle (0.1cm);
\node at (7.25,-0) {(MOY IIa)};
\end{tikzpicture}

\vspace{.2in}

\begin{tikzpicture} \draw[->](-1.8,0.2)--(-0.8,1.2);\draw[->](-0.8,0.2)--(-1.8,1.2);\draw[fill=black] (-1.3,0.7) circle (0.1cm); \draw[->](-0.8,1.2)--(0.2,0.2);\draw[->](0.2,1.2)--(-0.8,0.2);\draw[fill=black] (-0.3,0.7) circle (0.1cm); \node at (0.6,0.65) {$=$}; \draw[->] (1,0.2)--(2,1.2); \draw[<-] (1,1.2)--(2,0.2); \draw (1.5,0.7) circle (0.1cm); \draw[<-] (2,0.2)--(3,1.2); \draw[->] (2,1.2)--(3,0.2); \draw[fill=black] (2.5,0.7) circle (0.1cm); \node at (3.4,0.65) {$+$}; \node at (4.4,0.65) {$\dfrac{q^2+aq^6}{1-q^2}$}; \draw[->] plot [smooth] coordinates {(5.2,0.2)(5.4,0.7)(5.2,1.2)}; \draw[<-] plot [smooth] coordinates {(6.2,0.2)(6,0.7)(6.2,1.2)}; \node at (2.5,-0.5) {(MOY IIb*)};

\draw[->](7.5,0)--(8,0.5);\draw[->](8,0)--(7.5,0.5);\draw[->](8.5,0)--(8.5,0.5); \draw[fill=black] (7.75,0.25) circle (0.075cm); \draw[->](7.5,0.5)--(7.5,1);\draw[->](8,0.5)--(8.5,1);\draw[->] (8.5,0.5)--(8,1); \draw[fill=black] (8.25,0.75) circle (0.075cm); \draw[->](7.5,1)--(8,1.5);\draw[->](8,1)--(7.5,1.5);\draw[->] (8.5,1)--(8.5,1.5); \draw[fill=black] (7.75,1.25) circle (0.075cm); \node at (9.2,0.75) {$+\,\, q^2$}; \draw[->] (9.8,0)--(9.8,1.5); \draw[->](10.3,0)--(10.8,1.5); \draw[->](10.8,0)--(10.3,1.5); \draw[fill=black] (10.55,0.75) circle (0.1cm); \node at (11.3,0.75) {$=\,\, q^2$}; \draw[->] (11.6,0)--(12.1,1.5); \draw[->](12.1,0)--(11.6,1.5); \draw[->](12.6,0)--(12.6,1.5); \draw[fill=black] (11.85,0.75) circle (0.1cm); \node at (12.9,0.75) {$+$}; \draw[->](13.4,0)--(13.4,0.5);\draw[->](13.9,0)--(14.4,0.5);\draw[->] (14.4,0)--(13.9,0.5); \draw[fill=black] (14.15,0.25) circle (0.075cm); \draw[->](13.4,0.5)--(13.9,1);\draw[->](13.9,0.5)--(13.4,1);\draw[->] (14.4,0.5)--(14.4,1); \draw[fill=black] (13.65,0.75) circle (0.075cm); \draw[->](13.4,1)--(13.4,1.5);\draw[->](13.9,1)--(14.4,1.5);\draw[->] (14.4,1)--(13.9,1.5); \draw[fill=black] (14.15,1.25) circle (0.075cm); \node at (11.2,-0.5) {(MOY III)}; \end{tikzpicture}

\caption{Relations for $P^{b}(D)$. In this figure we omit the notation $P^{b}(\cdot)$
for readability.\label{fig:IntroPB}}

\end{figure}

The relations in Figure \ref{fig:IntroPB} are not always enough to
determine $P^{b}(D)$, though in many examples they do suffice. In
Section \ref{sec:Decat} we use $P^{b}(D)$ to show that $\HH(D)$
is not an invariant of virtual links, even when presented as a virtual
braid closure, by showing it violates the virtual exchange move. 

Recent research involving annular link homology by Auroux-Grigsby-Wehrli
in the $\mathfrak{sl}_{2}$ case \cite{AGW15} and Queffelec-Rose
in the $\mathfrak{sl}_{n}$ case \cite{QR15} give some insight into
why to expect this issue. Annular link homology theories, that is
homology theories of closed braids in the thickened annulus $D^{2}\times S^{1}$,
are normally constructed via the use of Hochschild homology on chain
complexes of bimodules associated to braids. Hochschild homology $\mbox{HH}(C)$
acts as a trace on the homotopy category of bimodules, but in general
does not act like a Markov trace. In particular, if $\beta_{1}$ and
$\beta_{2}$ are two braids which are Markov equivalent and $C(\beta_{1})$
and $C(\beta_{2})$ are their associated chain complexes of bimodules,
then $\mbox{HH}(C(\beta_{1}))$ is not necessarily homotopy equivalent
to $\mbox{HH}(C(\beta_{2}))$. This corresponds to the fact that even
though the braid closures of $\beta_{1}$ and $\beta_{2}$ are isotopic
as links in $S^{3},$ they may not be isotopic in $D^{2}\times I.$ 

However, $\HH(D)$ is not quite an annular invariant. Though $\HH(D)$
can be constructed using Hochschild homology (see \cite{Kh07}), it
is invariant under the second Markov move. This is why $\HH(D)$ gives
invariants of links in $S^{3}$ when $D$ is presented as the closure
of a braid. Though it does show some behavior of an annular invariant
as well. The Reidemeister IIb configuration $\posxnegxd$ can only
appear in a braid closure when the braid axis is between the two strands.
In the case of annular invariants the braid axis is an obstruction
to isotopy, and disallows the isotopy $\posxnegxd\sim\arcsacross$.
In an annular invariant the exchange move $\posxnegxd\sim\negxposxd$
is disallowed, however this move preserves the isomorphism type of
$\HH(D).$

\subsection*{Outline of the paper. }

In Section \ref{sec:HOMFLYPoly} we review the definition of the HOMFLY-PT
polynomial and the MOY construction of the HOMFLY-PT polynomial. We
use nonstandard conventions in this text to illuminate the connections
with HOMFLY-PT homology. In Section \ref{sec:HOMFLYHomology} we review
the construction of HOMFLY-PT homology of links using closed braid
diagrams. We also review some homological algebra, in particular properties
of Koszul complexes. In Section \ref{sec:GenDia} we explore the properties
of HOMFLY-PT homology for general link diagrams. We introduce the
role of virtual crossings in this framework and use virtual crossings
as a tool to prove that HOMFLY-PT homology is not invariant under
the Reidemeister IIb move. Finally, in Section \ref{sec:Decat} we
explore the decategorification (Poincaré series) of HOMFLY-PT homology
and use it to prove that HOMFLY-PT homology cannot be extended to
an invariant of virtual links. We also include in Appendix \ref{sec:Virtual}
a description of the framework of virtual crossing filtrations and
the necessary homological algebra.

\subsection*{Acknowledgements. }

The author would like to thank Mikhail Khovanov, Lenny Ng, and Lev
Rozansky for many helpful conversations and their feedback. The author
would like to also thank Matt Hogancamp for encouraging him to further
explore an observation which eventually became the text here.

\section{The MOY construction of the HOMFLY-PT polynomial\label{sec:HOMFLYPoly}}

We begin by recalling two constructions of the HOMFLY-PT polynomial
for oriented links. Most of this material is well-known, but we introduce
it with the purpose of setting our conventions for the sequel. The
first construction is given by a skein relation and first appeared
in \cite{HOMFLY}. The second construction, first introduced by Murakami,
Ohtsuki, and Yamada in \cite{MOY}, constructs the HOMFLY-PT polynomial
in terms of a state sum formula. It is this second construction which
is categorified in the construction of Khovanov and Rozansky's HOMFLY-PT
homology.

\subsection{The HOMFLY-PT polynomial of an oriented link diagram}

Let $L$ denote a link in $\mathbb{\mathbb{R}}^{3}$. In this text
we will assume all links are oriented. Let $D$ denote a link diagram
of $L$, that is a regular projection of $L$ onto a copy of $\mathbb{R}^{2}$.
The HOMFLY-PT polynomial is an invariant of links which takes (oriented)
link diagrams to elements of $\mathbb{Z}(q,a)$.
\begin{defn}
\label{def:HOMFLYpoly} Let $D$ be a link diagram and let $O$ be
a simple closed curve in the plane of the link diagram. We define
the \emph{HOMFLY-PT polynomial, }$P(D)\in\mathbb{Z}(q,a)$, via the
following relations:\end{defn}
\begin{enumerate}
\item $P(\emptyset)=1,\,\,P(O)=\dfrac{1+aq^{2}}{1-q^{2}}$
\item $P(D\sqcup O)=P(D)P(O)$
\item $qP(D_{+})-q^{-1}P(D_{-})=(q-q^{-1})P(D_{0})$, where $D_{+},D_{-},\mbox{ and }D_{0}$
are link diagrams which are the same except in the neighborhood of
a single point where $D_{+}=\posx,$ $D_{-}=\negx,$ and $D_{0}=\arcs$.
\end{enumerate}
We will call a crossing which locally looks like $\posx$ a \emph{positive
}crossing\emph{, }and a crossing that locally looks like $\negx$
a \emph{negative crossing. }Let $f,g\in\mathbb{Z}(q,a)$ be nonzero.
We will write $f\dot{=}g$ if $f=(-1)^{i}a^{j}q^{k}g$ for some $i,j,k\in\mathbb{Z}$.
In other words, we write $f\dot{=}g$ if $f/g$ is a unit in $\mathbb{Z}[q^{\pm1},a^{\pm1}]$.
\begin{thm}
[HOMFLY \cite{HOMFLY} , PT \cite{PT}]\label{thm:HOMFLYPoly}Let
$D$ and $D'$ be two link diagrams of a link $L$. Then $P(D)\dot{=}P(D')$.
Furthermore, $P(D_{2})=-q^{-2}P(D_{1})$ and $P(D_{3})=$ $aq^{2}P(D_{1}),$
where $D_{1},D_{2}\mbox{, and }D_{3}$ are link diagrams which are
the same except in the neighborhood of a single point where they are
as in Figure \ref{fig:Reid1Shift}
\end{thm}
We will often denote the HOMFLY-PT polynomial of a link by $P(L)$,
suppressing the choice of link diagram. In this case $P(L)$ is well-defined
up to a unit in $\mathbb{Z}[q^{\pm1},a^{\pm1}]$.

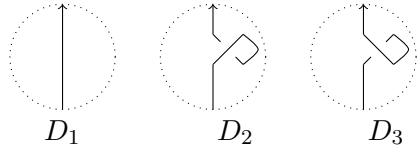
\begin{figure}[h]
\begin{tikzpicture}

\draw[->] (0,-.2)--(0,1.2);
\node at (0,-0.5) {$D_1$};
\draw[dotted] (0,0.5) circle (0.7cm);

\draw (2,-.2)--(2,0.4);
\draw (2,0.4)--(2.4,0.8);
\draw plot [smooth] coordinates {(2.4,0.8)(2.6,0.6)(2.4,0.4)};
\draw (2.4,0.4)--(2.3,0.5);
\draw (2.1,0.7)--(2,0.8);
\draw[->] (2,0.8)--(2,1.2);
\node at (2.3,-0.5) {$D_2$};
\draw[dotted] (2,0.5) circle (0.7cm);

\draw (4,-.2)--(4,0.4);
\draw (4,0.4)--(4.1,0.5);
\draw (4.3,0.7)--(4.4,0.8);
\draw plot [smooth] coordinates {(4.4,0.8)(4.6,0.6)(4.4,0.4)};
\draw (4.4,0.4)--(4,0.8);
\draw[->] (4,0.8)--(4,1.2);
\node at (4.3,-0.5) {$D_3$};
\draw[dotted] (4,0.5) circle (0.7cm);

\end{tikzpicture}

\caption{The diagrams $D_{1},D_{2},\mbox{ and \ensuremath{D_{3}}}$ from Theorem
\ref{thm:HOMFLYPoly} \label{fig:Reid1Shift}}
\end{figure}

\begin{rem}
We are using non-standard conventions for the HOMFLY-PT polynomial
in this text. The HOMFLY-PT polynomial as defined here is not a polynomial,
but rather is a rational function. One may choose a different normalization
where both $P(D)\in\mathbb{Z}[a^{\pm1},q^{\pm1}]$ is honestly a (Laurent)
polynomial and $P(D_{1})=P(D_{2})=P(D_{3})$ where $D_{1},D_{2},\mbox{ and \ensuremath{D_{3}}}$
are as in Figure \ref{fig:Reid1Shift}. The choice of normalization
here coincides with our conventions for HOMFLY-PT homology in the
sequel.
\end{rem}

\subsection{The MOY construction of the HOMFLY-PT polynomial}

Murakami, Ohtsuki, and Yamada in \cite{MOY} give a construction of
the $\mathfrak{sl}_{n}$ polynomial, $P_{n}(L)\in\mathbb{Z}[q,q^{-1}]$,
of a link $L$ using evaluations of oriented colored trivalent plane
graphs. These trivalent plane graphs correspond to the intertwiners
between tensor powers of fundamental representations of $\mathcal{U}_{q}(\mathfrak{sl}_{n})$.
The $\mathfrak{sl}_{n}$ polynomial is actually a specialization of
the HOMFLY-PT polynomial, that is $P_{n}(L)(q)=P(L)(q,a=q^{2-2n})$
(in our conventions). We may adjust the MOY construction of the $\mathfrak{sl}_{n}$
polynomial to compute the HOMFLY-PT polynomial. We will replace the
``wide edge'' graph of \cite{MOY} with a single degree $4$ vertex
(see Figure \ref{fig:MOYWideEdge}) which we will call a \emph{MOY
vertex.}

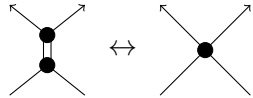
\begin{figure}[h]
\begin{tikzpicture}
\draw (0,0)--(0.5,0.4);
\draw (1,0)--(0.5,0.4);
\draw[fill=black] (0.5,0.4) circle (0.1cm);
\draw (0.45,0.4)--(0.45,0.8);
\draw (0.55,0.4)--(0.55,0.8);
\draw[fill=black] (0.5,0.8) circle (0.1cm);
\draw[->] (0.5,0.8)--(0,1.2);
\draw[->] (0.5,0.8)--(1,1.2);

\node at (1.5,0.6) {$\leftrightarrow$};

\draw[->] (2,0)--(3.2,1.2);
\draw[->] (3.2,0)--(2,1.2);
\draw[fill=black] (2.6,0.6) circle (0.1cm);
\end{tikzpicture}

\caption{The MOY wide edge graph and our MOY vertex \label{fig:MOYWideEdge}}

\end{figure}

The MOY state model of the HOMFLY-PT polynomial writes a link diagram
as a formal $\mathbb{Z}(q,a)$-linear combination of planar, oriented,
$4$-regular graphs. The orientation locally at each vertex is the
same as the orientation of the MOY vertex in Figure \ref{fig:MOYWideEdge}.
We call such planar, oriented, $4$-regular graphs \emph{MOY graphs.}

We now define the MOY construction of the HOMFLY-PT polynomial. Let
$D$ be a link diagram. We can resolve any crossing $c$ into either
an oriented smoothing $\arcs$ or a MOY vertex $\wideedge$(with consistent
orientation). To each resolution of $c$ we associate a \emph{weight}.
If we smooth the crossing then the resolution has weight $0$. If
we replace the crossing with a MOY vertex, then the weight is $-2$
if the crossing was positive and 0 if the crossing was negative. A
resolution chart is given in Figure \ref{fig:MOYRelations} for reference.
We define a \emph{state} $\sigma$ of $D$ as a choice of resolution
for every crossing $c$ in $D$. If $D$ has $n$ crossings, then
it has $2^{n}$ possible states. We define the \emph{weight of a state}
$\mu(\sigma)$ to be the sum of the weights of the chosen resolutions
of $\sigma$. Finally we will set $\nu(\sigma)$ to be the number
of MOY vertices in $\sigma.$

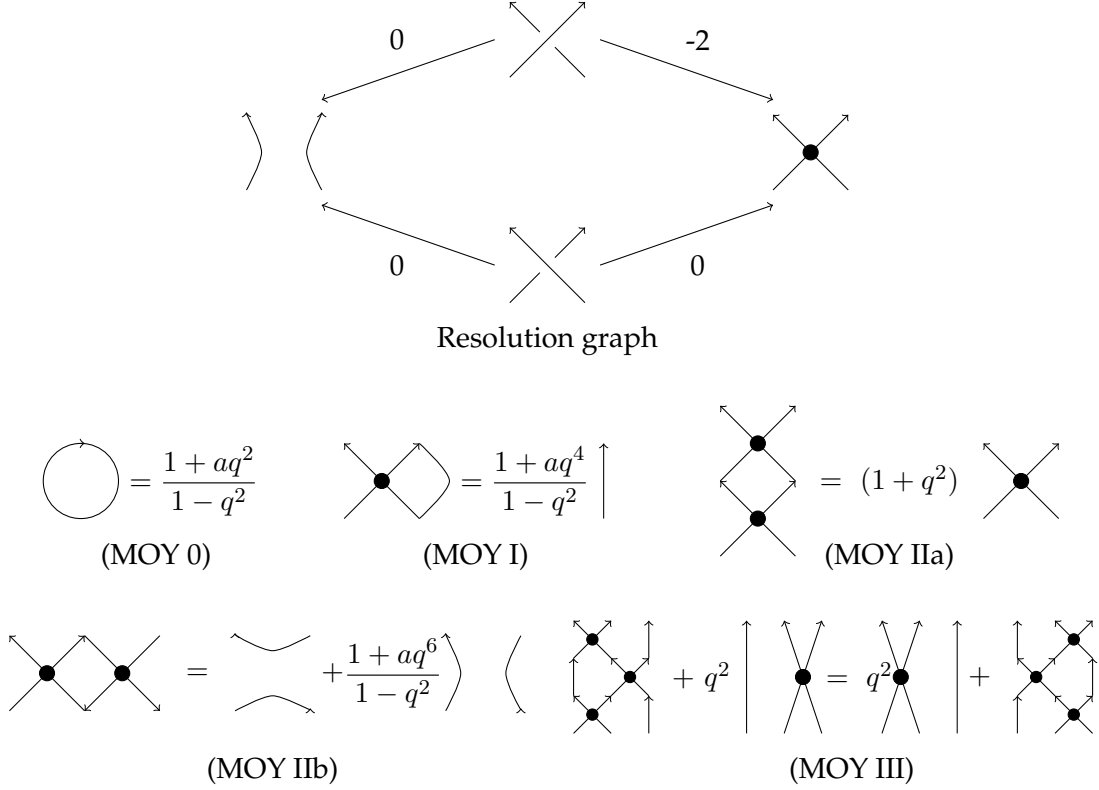
\begin{figure}[h]
\begin{tikzpicture}
\draw[->] plot [smooth] coordinates {(-.5,2)(-0.3,2.5)(-0.5,3)}; \draw[->] plot [smooth] coordinates {(.5,2)(0.3,2.5)(.5,3)};
\draw[->] (3,3.5)--(4,4.5); \draw (4,3.5)--(3.6,3.9); \draw[->] (3.4,4.1)--(3,4.5);
\draw[->] (4,0.5)--(3,1.5); \draw (3,0.5)--(3.4,0.9); \draw[->] (3.6,1.1)--(4,1.5);
\draw[->] (6.5,2)--(7.5,3); \draw[->] (7.5,2)--(6.5,3); \draw[fill=black] (7,2.5) circle (0.1cm);
\draw[<-] (0.5,1.8)--(2.8,1); \node at (1.5,1) {0}; \draw[<-] (0.5,3.2)--(2.8,4); \node at (1.5,4) {0}; \draw[->] (4.2,4)--(6.5,3.2); \node at (5.5,1) {0}; \draw[->] (4.2,1)--(6.5,1.8); \node at (5.5,4) {-2};
\node at (3.5,0) {Resolution graph};
\end{tikzpicture}
\vspace{.2in}

\begin{tikzpicture} 
\draw (-3.5,1) circle (.5cm);
\draw[->] (-3.5,1.5)--(-3.45,1.5);
\node at (-2,1) {$=\dfrac{1+aq^2}{1-q^2}$}; \node at (-2.5,0) {(MOY 0)};
\draw[->](0,0.5)--(1,1.5); \draw[->](1,0.5)--(0,1.5); \draw[fill=black](0.5,1) circle (0.1cm); \draw plot [smooth] coordinates {(1,1.5)(1.4,1)(1,0.5)}; \node at (2.4,1) {$= \dfrac{1+aq^4}{1-q^2}$}; \draw[->] (3.45,0.5)--(3.45,1.5);
\node at (1.75,-0) {(MOY I)};
\draw[->](5,0)--(6,1); \draw[->](6,0)--(5,1); \draw[fill=black](5.5,0.5) circle (0.1cm); \draw[->](5,1)--(6,2); \draw[->](6,1)--(5,2); \draw[fill=black](5.5,1.5) circle (0.1cm); \node at (7.25,1) {$=\,\,(1+q^2)$};
\draw[->](8.5,0.5)--(9.5,1.5); \draw[->](9.5,0.5)--(8.5,1.5); \draw[fill=black](9,1) circle (0.1cm);
\node at (7.25,-0) {(MOY IIa)};
\end{tikzpicture}

\vspace{.2in}

\begin{tikzpicture} 
\draw[->](0,0.3)--(1,1.3);\draw[->](1,0.3)--(0,1.3);\draw[fill=black] (0.5,0.8) circle (0.1cm);
\draw[->](1,1.3)--(2,0.3);\draw[->](2,1.3)--(1,0.3);\draw[fill=black] (1.5,0.8) circle (0.1cm);
\node at (2.5,0.8) {$=$}; 
\draw[->] plot [smooth] coordinates {(3,0.3)(3.5,0.5)(4,0.3)}; \draw[<-] plot [smooth] coordinates {(3,1.3)(3.5,1.1)(4,1.3)}; \node at (4.3,0.8) {$+$}; \node at (5.1,0.8) {$\dfrac{1+aq^6}{1-q^2}$}; \draw[->] plot [smooth] coordinates {(5.8,0.3)(6,0.8)(5.8,1.3)}; \draw[<-] plot [smooth] coordinates {(6.8,0.3)(6.6,0.8)(6.8,1.3)}; \node at (3.5,-0.5) {(MOY IIb)};

\draw[->](7.5,0)--(8,0.5);\draw[->](8,0)--(7.5,0.5);\draw[->](8.5,0)--(8.5,0.5);
\draw[fill=black] (7.75,0.25) circle (0.075cm);
\draw[->](7.5,0.5)--(7.5,1);\draw[->](8,0.5)--(8.5,1);\draw[->] (8.5,0.5)--(8,1);
\draw[fill=black] (8.25,0.75) circle (0.075cm);
\draw[->](7.5,1)--(8,1.5);\draw[->](8,1)--(7.5,1.5);\draw[->] (8.5,1)--(8.5,1.5);
\draw[fill=black] (7.75,1.25) circle (0.075cm);
\node at (9.2,0.75) {$+\,\, q^2$};
\draw[->] (9.8,0)--(9.8,1.5); \draw[->](10.3,0)--(10.8,1.5); \draw[->](10.8,0)--(10.3,1.5);
\draw[fill=black] (10.55,0.75) circle (0.1cm);
\node at (11.3,0.75) {$=\,\, q^2$};
\draw[->] (11.6,0)--(12.1,1.5); \draw[->](12.1,0)--(11.6,1.5); \draw[->](12.6,0)--(12.6,1.5);
\draw[fill=black] (11.85,0.75) circle (0.1cm);
\node at (12.9,0.75) {$+$}; 
\draw[->](13.4,0)--(13.4,0.5);\draw[->](13.9,0)--(14.4,0.5);\draw[->] (14.4,0)--(13.9,0.5);
\draw[fill=black] (14.15,0.25) circle (0.075cm);
\draw[->](13.4,0.5)--(13.9,1);\draw[->](13.9,0.5)--(13.4,1);\draw[->] (14.4,0.5)--(14.4,1);
\draw[fill=black] (13.65,0.75) circle (0.075cm);
\draw[->](13.4,1)--(13.4,1.5);\draw[->](13.9,1)--(14.4,1.5);\draw[->] (14.4,1)--(13.9,1.5);
\draw[fill=black] (14.15,1.25) circle (0.075cm); 
\node at (11.2,-0.5) {(MOY III)};
\end{tikzpicture}
\caption{Resolution chart and MOY relations. In this figure we omit the notation
$\bar{P}(\cdot)$ for readability.\label{fig:MOYRelations}}

\end{figure}

\begin{defn}
\label{def:MOYPoly}The \emph{MOY polynomial, }$\bar{P}(D)$, is given
by 
\begin{equation}
\bar{P}(D)=\sum_{\sigma}(-1)^{\nu(\sigma)}q^{\mu(\sigma)}\bar{P}(D_{\sigma})\label{eq:MOYpoly}
\end{equation}

Where $\bar{P}(D_{\sigma})$ is the evaluation of the MOY graph given
by the state $\sigma$ using the relations in Figure \ref{fig:MOYRelations}.\end{defn}
\begin{thm}
[ Murakami-Ohtsuki-Yamada \cite{MOY}]\label{thm:MOYPoly} The relations
given in Figure \ref{fig:MOYRelations} are sufficient to compute
$\bar{P}(D_{\sigma})$ as an element of $\mathbb{Z}(q,a)$ for any
link diagram $D$ and any state $\sigma.$ Furthermore, $\bar{P}(D)=P(D)$
for any link diagram.\end{thm}
\begin{example}
\label{exa:LHTrefoil}As an example we will work out $\bar{P}(D)$
for the diagram of the left-handed trefoil knot given in Figure \ref{fig:negTrefoil}.
We resolve the diagram into its eight resolutions. Many of the resolutions
give the same graph due to symmetries in the diagram. Therefore, 
\[
\bar{P}(D)=\bar{P}(\Gamma_{0})-3\bar{P}(\Gamma_{1})+3\bar{P}(\Gamma_{2})-\bar{P}(\Gamma_{3}),
\]
where the graphs $\Gamma_{i}$ are shown in Figure \ref{fig:negTrefoil}.

\begin{figure}[h]
\scalebox{0.8}{
\begin{tikzpicture}
\draw (1,0)--(0,1); \draw (0,0)--(0.4,0.4); \draw (0.6,0.6)--(1,1);
\draw (1,1)--(0,2); \draw (0,1)--(0.4,1.4); \draw (0.6,1.6)--(1,2);
\draw[->] (1,2)--(0,3); \draw (0,2)--(0.4,2.4); \draw[->] (0.6,2.6)--(1,3);
\draw plot [smooth] coordinates {(0,3)(-0.5,3.25)(-0.75,1.5)(-0.5,-0.25)(0,0)};
\draw plot [smooth] coordinates {(1,3)(1.5,3.25)(1.75,1.5)(1.5,-0.25)(1,0)};
\node at (0.5,-0.5) {$D$};

\draw[->] (3,0)--(3,3); \draw [->] (4,0)--(4,3);
\draw plot [smooth] coordinates {(3,3)(2.5,3.25)(2.25,1.5)(2.5,-0.25)(3,0)};
\draw plot [smooth] coordinates {(4,3)(4.5,3.25)(4.75,1.5)(4.5,-0.25)(4,0)};
\node at (3.5,-0.5) {$\Gamma_0$};

\draw (6,0)--(6,1); \draw  (7,0)--(7,1);
\draw[->] (6,1)--(7,2);\draw[->] (7,1)--(6,2); \draw[fill=black] (6.5,1.5) circle (0.1cm);
\draw[->] (6,2)--(6,3); \draw[->]  (7,2)--(7,3);
\draw plot [smooth] coordinates {(6,3)(5.5,3.25)(5.25,1.5)(5.5,-0.25)(6,0)};
\draw plot [smooth] coordinates {(7,3)(7.5,3.25)(7.75,1.5)(7.5,-0.25)(7,0)};
\node at (6.5,-0.5) {$\Gamma_1$};

\draw[->] (9,0)--(10,1);\draw[->] (10,0)--(9,1); \draw[fill=black] (9.5,0.5) circle (0.1cm);
\draw[->] (9,1)--(10,2);\draw[->] (10,1)--(9,2); \draw[fill=black] (9.5,1.5) circle (0.1cm);
\draw (9,2)--(9,3); \draw  (10,2)--(10,3);
\draw plot [smooth] coordinates {(9,3)(8.5,3.25)(8.25,1.5)(8.5,-0.25)(9,0)};
\draw plot [smooth] coordinates {(10,3)(10.5,3.25)(10.75,1.5)(10.5,-0.25)(10,0)};
\node at (9.5,-0.5) {$\Gamma_2$};

\draw[->] (12,0)--(13,1);\draw[->] (13,0)--(12,1); \draw[fill=black] (12.5,0.5) circle (0.1cm);
\draw[->] (12,1)--(13,2);\draw[->] (13,1)--(12,2); \draw[fill=black] (12.5,1.5) circle (0.1cm);
\draw[->] (12,2)--(13,3);\draw[->] (13,2)--(12,3); \draw[fill=black] (12.5,2.5) circle (0.1cm);
\draw plot [smooth] coordinates {(12,3)(11.5,3.25)(11.25,1.5)(11.5,-0.25)(12,0)};
\draw plot [smooth] coordinates {(13,3)(13.5,3.25)(13.75,1.5)(13.5,-0.25)(13,0)};
\node at (12.5,-0.5) {$\Gamma_3$};

\end{tikzpicture}}
\caption{Diagram of the left-handed trefoil knot and its resolutions \label{fig:negTrefoil}}

\end{figure}
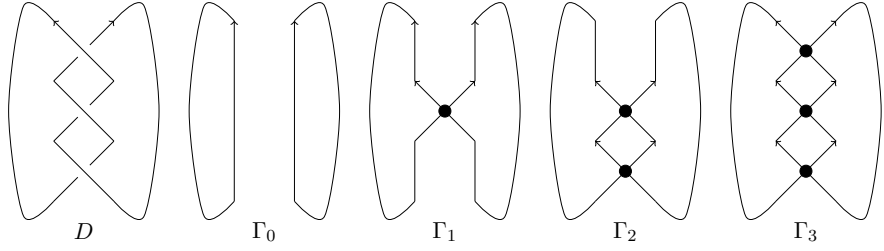

Using (MOY II), we see that $\bar{P}(\Gamma_{2})=(1+q^{2})\bar{P}(\Gamma_{1})$
and $\bar{P}(\Gamma_{3})=(1+q^{2})\bar{P}(\Gamma_{2})=(1+q^{2})^{2}\bar{P}(\Gamma_{1}).$
Using (MOY I) and (MOY 0) we see that 
\[
\bar{P}(\Gamma_{1})=\frac{1+aq^{4}}{1-q^{2}}\bar{P}(O)=\frac{(1+aq^{2})(1+aq^{4})}{(1-q^{2})^{2}},\,\,\,\bar{P}(\Gamma_{0})=\left(\frac{1+aq^{2}}{1-q^{2}}\right)^{2}.
\]

Therefore, combining everything, we compute
\[
\bar{P}(D)=\frac{1+aq^{2}}{1-q^{2}}\left(\frac{1+aq^{2}}{1-q^{2}}-\frac{3(1+aq^{4})}{1-q^{2}}+\frac{3(1+q^{2})(1+aq^{4})}{1-q^{2}}-\frac{(1+q^{2})^{2}(1+aq^{4})}{1-q^{2}}\right)
\]

which simplifies to 
\begin{equation}
\bar{P}(D)=\left(\frac{1+aq^{2}}{1-q^{2}}\right)(q^{2}+aq^{2}+aq^{6}).\label{eq:negTreMOY}
\end{equation}

\end{example}

\section{HOMFLY-PT homology for closed braid diagrams\label{sec:HOMFLYHomology}}

In this section we introduce the construction of Khovanov and Rozansky's
HOMFLY-PT homology. The approach of this construction is to associate
a chain complex of modules to every MOY graph and a \emph{bicomplex
}of modules to every link diagram. Our approach in this section is
most similar to the approach of Rasmussen in \cite{Ras06} where we
ignore his ``$\mathfrak{sl}_{n}$'' differential, as it is not needed
in the construction of HOMFLY-PT homology.

\subsection{Koszul complexes}

Before introducing HOMFLY-PT homology, we recall some terminology
and notation involving Koszul complexes. Let $R=\bigoplus_{i\in\mathbb{Z}}R_{i}$
be a $\mathbb{Z}$-graded commutative $\mathbb{Q}$-algebra and $M=\bigoplus_{i\in\mathbb{Z}}M_{i}$
be a $\mathbb{Z}$-graded $R$-module. It will be instructive to keep
the example of $R=\mathbb{Q}[\mathbf{x},\mathbf{y}]$ in mind, where
$\mathbf{x}$ and $\mathbf{y}$ are finite lists of variables (not
necessarily of the same length). We define the grading shift functor
$\bullet(k)$ by $M(k)_{j}=M_{j-k}$ for all $j\in\mathbb{Z}$. We
will commonly use a nonstandard notation for grading shifts. In particular,
we will set $q^{k}M:=M(k)$ and say $\deg_{q}(x)=q^{j}$ if $x\in M_{j}$.
\begin{defn}
Let $p\in R$ be an element of degree $k$. The \emph{Koszul complex
}of $p$ is defined as the chain complex 
\[
[p]_{R}=q^{k}R_{1}\xrightarrow{\,\,p\,\,}R_{0},
\]

where $p$ is used to denote the algebra endomorphism of $R$ given
by multiplication by $p$. Here $R_{0}=R_{1}=R$ and the subscript
is simply used to denote the homological degree of the module. We
will often write $[p]=[p]_{R}$ when there can be no confusion. Now
let $\mathbf{p}=p_{1},...,p_{k}$ be a sequence of elements in $R$.
Then we define the \emph{Koszul complex }of $\mathbf{p}$ as the complex
\[
\begin{bmatrix}p_{1}\\
\vdots\\
p_{k}
\end{bmatrix}=[p_{1}]\otimes_{R}\cdots\otimes_{R}[p_{k}]
\]

where $\otimes_{R}$ denotes the ordinary tensor product of chain
complexes. 
\end{defn}
As a convention, we will call the homological grading in Koszul complexes
the \emph{Hochschild grading} and denote it by $\deg_{a}$. We write
$\deg_{a}(x)=a^{k}$ to say that $x$ is in Hochschild degree $k$
and similarly write $a^{k}M$ to denote that $M$ is being shifted
$k$ in Hochschild degree. 

We say a sequence of elements $\mathbf{p}=p_{1},...,p_{k}$ in $R$
is a \emph{regular sequence }if $p_{m}$ is not a zero divisor in
$R/(p_{1},...,p_{m-1})$ for all $m=1,...,k$. The following proposition
is a standard fact in homological algebra and is proven in many introductory
texts such as \cite{Wei94}.
\begin{prop}
Let $\mathbf{p}=p_{1},...,p_{n}$ be a regular sequence in $R$. Then
the Koszul complex of $\mathbf{p}$ is a graded free $R$-module resolution
of $R/(p_{1},...,p_{n})$.
\end{prop}
The notation we use for Koszul complexes is reminiscent of the notation
for a row vector in $R^{\oplus n}$. Note that we will always use
square brackets for Koszul complexes and round brackets for row vectors
in $R^{\oplus n}$ to eliminate any confusion. Along these lines,
we can look at ``row operations'' on Koszul complexes.
\begin{prop}
Let $\mathbf{p}=p_{1},...,p_{k}$ be a sequence of elements in $R$,
and let $\lambda\in\mathbb{Q}$, then
\[
\begin{bmatrix}\vdots\\
p_{i}\\
\vdots\\
p_{j}\\
\vdots
\end{bmatrix}\simeq\begin{bmatrix}\vdots\\
p_{i}+\lambda p_{j}\\
\vdots\\
p_{j}\\
\vdots
\end{bmatrix}
\]

A homotopy equivalence of this form will be called a \emph{change
of basis.}\end{prop}
\begin{proof}
We will omit grading shifts in the proof for clarity. We consider
the map $\Phi:[p_{i}]\otimes_{R}[p_{j}]\to[p_{i}+\lambda p_{j}]\otimes_{R}[p_{j}]$
given by 

\begin{center}
\begin{diagram}
[p_{i}]\otimes_{R}[p_{j}]=&R & \rTo{\begin{pmatrix}p_{i}\\p_{j}\end{pmatrix}} &R \oplus R & \rTo{(-p_j \,\, p_i)} & R \\
\dTo{\Phi}&\dTo{1}&    &\dTo{\begin{pmatrix}1 & \lambda\\0 & 1\end{pmatrix}}&             &\dTo{1}\\
[p_{i}+\lambda p_j]\otimes_{R}[p_{j}]=&R & \rTo_{\begin{pmatrix}p_{i}+\lambda p_j\\p_{j}\end{pmatrix}} &R \oplus R & \rTo_{(-p_j \,\, p_i+\lambda p_j)} & R \\
\end{diagram}
\end{center}

This map is clearly invertible.
\end{proof}

\subsection{Marked MOY graphs}

A \emph{marked MOY graph }is a MOY graph $\Gamma$ (possibly with
boundary) with markings such that the marks partition the graph into
some combination of \emph{elementary MOY graphs }as shown in Figure
\ref{fig:elemMOY}. We label the marks and the endpoints of the graph
(if any) with variables. Typically, though not necessarily, we will
label outgoing edges by variables $y_{i}$, incoming edges by variables
$x_{i}$, and internal marks by variables $t_{i}.$ An example of
this process is given in Figure \ref{fig:elemMOY} .

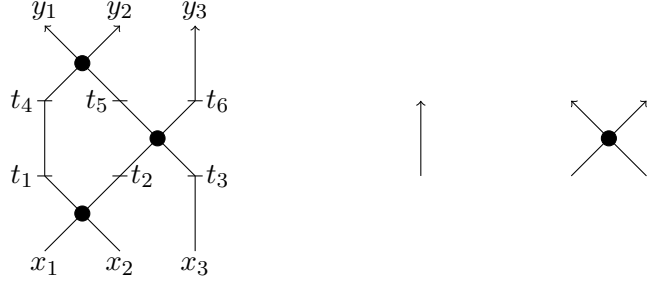
\begin{figure}[h]
\begin{tikzpicture}
\draw (0,0)--(1,1); \draw (1,0)--(0,1); \draw (2,0)--(2,1);\draw[fill=black] (0.5,0.5) circle (0.1cm);
\draw (0,1)--(0,2); \draw (1,1)--(2,2); \draw (2,1)--(1,2);\draw[fill=black] (1.5,1.5) circle (0.1cm);
\draw[->] (0,2)--(1,3); \draw[->] (1,2)--(0,3); \draw[->] (2,2)--(2,3);
\draw[fill=black] (0.5,2.5) circle (0.1cm);
\draw (-0.1,1)--(0.1,1);\draw (0.9,1)--(1.1,1);\draw (1.9,1)--(2.1,1);
\draw (-0.1,2)--(0.1,2);\draw (0.9,2)--(1.1,2);\draw (1.9,2)--(2.1,2);
\node at (0,-0.2) {$x_1$};\node at (1,-0.2) {$x_2$};\node at (2,-0.2) {$x_3$};
\node at (0,3.2) {$y_1$};\node at (1,3.2) {$y_2$};\node at (2,3.2) {$y_3$};
\node at (-0.3,1) {$t_1$};\node at (1.3,1) {$t_2$};\node at (2.3,1) {$t_3$};
\node at (-0.3,2) {$t_4$};\node at (0.7,2) {$t_5$};\node at (2.3,2) {$t_6$};

\draw[->] (5,1)--(5,2);

\draw[->] (7,1)--(8,2);\draw[->](8,1)--(7,2);\draw[fill=black] (7.5,1.5) circle (0.1cm);
\end{tikzpicture}

\caption{An example of a marked MOY graph and the elementary MOY graphs\label{fig:elemMOY}}

\end{figure}

To a marked MOY graph $\Gamma$, we will associate a collection of
rings. Let $\mathbf{x},\mathbf{y},\mathbf{t}$ denote the lists of
incoming, outgoing, and internal variables respectively. We first
define the \emph{total ring} of $\Gamma$, $E^{t}(\Gamma)$ as the
polynomial ring $\mathbb{Q}[\mathbf{x},\mathbf{y},\mathbf{t}]$ containing
all variables. We make this ring into a graded ring by setting $\deg_{q}(x_{i})=\deg_{q}(y_{i})=\deg_{q}(t_{i})=q^{2}$.
We call this grading the \emph{internal }or \emph{quantum grading.
}We also suppose that all elements in $E^{t}(\Gamma)$ have Hochschild
degree $a^{0}$. The other rings we will define will be subrings of
$E^{t}(\Gamma).$ The \emph{edge ring, }$E(\Gamma)$, is the polynomial
ring of incoming and outgoing (``edge'') variables $\mathbb{Q}[\mathbf{x},\mathbf{y}]$.
$E^{t}(\Gamma)$ has a natural (possibly infinite rank) free $E(\Gamma)$-module
structure. We also define the \emph{incoming ring} (resp. \emph{outgoing
ring}) by $E^{i}(\Gamma)=\mathbb{Q}[\mathbf{x}]$ (resp. $E^{o}(\Gamma)=\mathbb{Q}[\yy]$).
Since $E(\Gamma)\cong E^{i}(\Gamma)\otimes_{\mathbb{Q}}E^{o}(\Gamma)$
as $\mathbb{Q}$-algebras, then any $E(\Gamma)$-module can be considered
as a $E^{i}(\Gamma)$-$E^{o}(\Gamma)$-bimodule. Note that if $\Gamma$
does not have any boundary (e.g; if it is a resolution of a link diagram),
then $E(\Gamma)\cong E^{i}(\Gamma)\cong E^{o}(\Gamma)\cong\mathbb{Q}$.
We list the four rings associated to $\Gamma$ in Figure \ref{fig:EdgeRings}
for easy reference.

\begin{figure}[h]
\begin{tabular}{|c||c|c|c|}
\hline 
Notation & Name & Ring & Variables Included\tabularnewline
\hline 
\hline 
$E^{t}(\Gamma)$ & Total ring & $\mathbb{Q}[\mathbf{x},\mathbf{y},\mathbf{t}]$ & incoming, outgoing, internal\tabularnewline
\hline 
$E(\Gamma)$ & Edge ring & $\mathbb{Q}[\mathbf{x},\mathbf{y}]$ & incoming, outgoing\tabularnewline
\hline 
$E^{i}(\Gamma)$ & Incoming ring & $\mathbb{Q}[\mathbf{x}]$ & incoming\tabularnewline
\hline 
$E^{o}(\Gamma)$ & Outgoing ring & $\mathbb{Q}[\mathbf{y}]$ & outgoing\tabularnewline
\hline 
\end{tabular}\caption{The four rings associated to a marked MOY graph $\Gamma$\label{fig:EdgeRings}}
\end{figure}

We will now define complexes $C(\Gamma)$ of free $E(\Gamma)$-modules
associated to a marked MOY graph $\Gamma.$ The chain modules of $C(\G)$
will be direct sums of shifted copies of $E^{t}(\Gamma)$. We do this
by first defining Koszul complexes associated to the elementary MOY
graphs and then give rules for how gluing the graphs together affects
the complexes associated to them. We will use the symbols $\arc$
and $\wideedge$ to denote the elementary arc and vertex MOY graphs.
To the arc, we associate the Koszul complex of $E^{t}(\arc)=E(\arc)=\mathbb{Q}[x,y]$-modules,
\begin{equation}
C(\arc)=[y-x]_{E(\arc)}=q^{2}aE(\arc)\xrightarrow{y-x}E(\arc)\label{eq:ArcComplex}
\end{equation}

and to the vertex graph, we associate the Koszul complex of $E^{t}(\wideedge)=E(\wideedge)=\mathbb{Q}[x_{1},x_{2},y_{1},y_{2}]$-modules,
\begin{equation}
\begin{aligned}C(\wideedge)=\begin{bmatrix}y_{1}+y_{2}-x_{1}-x_{2}\\
(y_{1}-x_{1})(y_{1}-x_{2})
\end{bmatrix}_{E(\wideedge)}= & q^{6}a^{2}E(\wideedge)\xrightarrow{A}q^{4}aE(\wideedge)\oplus q^{2}aE(\wideedge)\xrightarrow{B}E(\wideedge),\end{aligned}
\label{eq:VertexComplex}
\end{equation}

where 
\[
A=\begin{pmatrix}y_{1}+y_{2}-x_{1}-x_{2}\\
(y_{1}-x_{1})(y_{1}-x_{2})
\end{pmatrix},\,\,\,B=\begin{pmatrix}-(y_{1}-x_{1})(y_{1}-x_{2}) & y_{1}+y_{2}-x_{1}-x_{2}\end{pmatrix}
\]
Now suppose $\Gamma$ is a marked MOY graph with edge ring $E$ and
total ring $E^{t}$. Also let $\Gamma$' be another marked MOY graph
with edge ring $E'$ and total ring $E'^{t}$. The disjoint union
of these graphs $\Gamma\sqcup\Gamma'$ has edge ring $E''\cong E\otimes_{\mathbb{Q}}E'$
and total ring $E''^{t}\cong E^{t}\otimes_{\mathbb{Q}}E'^{t}$. To
the marked MOY graph $\Gamma\sqcup\Gamma'$ we will associated the
complex of $E''$-modules $C(\Gamma\sqcup\Gamma'):=C(\Gamma)\otimes_{\mathbb{Q}}C(\Gamma')$.
A picture of the corresponding diagram is shown in Figure \ref{fig:combiningMOYGraphs} 

\begin{figure}[h]
\begin{tikzpicture}

\draw(0,0)--(0,0.75); \draw(1,0)--(1,0.75);
\draw (-.2,0.75) rectangle (1.2,1.25); \node at (0.5,1) {$\Gamma$};
\draw[->](0,1.25)--(0,2); \draw[->](1,1.25)--(1,2);

\draw(2,0)--(2,0.75); \draw(3,0)--(3,0.75);
\draw (1.8,0.75) rectangle (3.2,1.25); \node at (2.5,1) {$\Gamma'$};
\draw[->](2,1.25)--(2,2); \draw[->](3,1.25)--(3,2);

\node at (1.5,-0.5) {$\Gamma\sqcup\Gamma'$};

\node at (5,1) {$\Gamma$}; \draw (5,1) circle (.5cm);
\node at (7,1) {$\Gamma'$}; \draw (7,1) circle (.5cm);
\draw[<-] plot [smooth] coordinates {(5.5,1.1)(6,1.3)(6.5,1.1)}; \draw (6,1.2)--(6,1.4);
\draw[->] plot [smooth] coordinates {(5.5,0.9)(6,0.7)(6.5,0.9)}; \draw (6,0.8)--(6,0.6);
\draw[->] (4,0)--(4.65,0.65); \draw[<-] (8,0)--(7.35,0.65);
\draw[<-] (4,2)--(4.65,1.35); \draw[->] (8,2)--(7.35,1.35);
\node at (6,1.6) {$z_1$}; \node at (6,0.4) {$z_2$};
\node at (6,-0.5) {$\Gamma\cup_\mathbf{z}\Gamma'$};
\end{tikzpicture}\caption{Examples of disjoint union and glueing of marked MOY graphs \label{fig:combiningMOYGraphs}}
\end{figure}
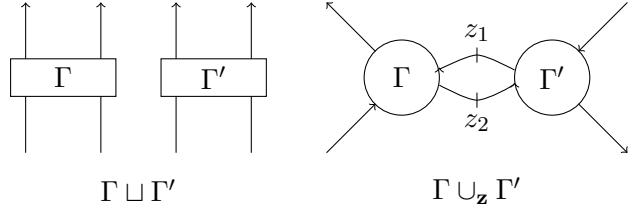

Finally, we define a complex for when we glue two marked MOY graphs
together. Let $\Gamma$ and $\Gamma'$ be two marked MOY graphs. We
can glue outgoing edges of $\Gamma$ to incoming edges of $\Gamma'$
(or vice versa) to get a new marked MOY graph. First suppose only
one endpoint from each graph are being glued together. Suppose that
endpoint in both $\Gamma$ and $\Gamma'$ is labeled by the variable
$z$ (that is, $z\in E^{i}(\Gamma)\cap E^{o}(\Gamma')$ or $z\in E^{i}(\Gamma')\cap E^{o}(\Gamma)$)
. Then we define the new graph $\Gamma\cup_{z}\Gamma'$ by identifying
the endpoints labeled by $z$ and associate to $\Gamma\cup_{z}\Gamma'$
the complex 
\begin{equation}
C(\Gamma\cup_{z}\Gamma'):=C(\Gamma)\otimes_{\mathbb{Q}[z]}C(\Gamma')\label{eq:MOYGluingRel}
\end{equation}

The edge ring of $\Gamma\cup_{z}\Gamma'$ is $E(\Gamma\cup_{z}\Gamma')=(E(\Gamma)\otimes_{\mathbb{Q}[z]}E(\Gamma'))/(z)$
and the total ring is $E^{t}(\Gamma\cup_{z}\Gamma')=E^{t}(\Gamma)\otimes_{\mathbb{Q}[z]}E^{t}(\Gamma')$.
Note that after gluing, $z$ is no longer in the edge ring as it is
an internal variable. We may glue multiple edges as once in a similar
manner. If $\mathbf{z}=z_{1},...,z_{n}$ are the variables at the
marked endpoints being identified, then we define $C(\Gamma\cup_{\mathbf{z}}\Gamma'):=C(\Gamma)\otimes_{\mathbb{Q}[\mathbf{z}]}C(\Gamma')$.
Similar to the case where we only identified one pair of edges, The
edge ring of $\Gamma\cup_{\zz}\Gamma'$ is $E(\Gamma\cup_{\zz}\Gamma')=(E(\Gamma)\otimes_{\mathbb{Q}[\zz]}E(\Gamma'))/(z_{1},...,z_{n})$
and the total ring is $E^{t}(\Gamma\cup_{\zz}\Gamma')=E^{t}(\Gamma)\otimes_{\mathbb{Q}[\zz]}E^{t}(\Gamma')$. 

We can also describe disjoint union and gluing of marked MOY graphs
in terms of Koszul complexes. Suppose $C(\Gamma)$ and $C(\Gamma')$
are given by the Koszul complexes
\[
C(\Gamma)=\begin{bmatrix}p_{1}\\
\vdots\\
p_{m}
\end{bmatrix}_{E^{t}(\Gamma)}\mbox{ and }\,C(\Gamma')=\begin{bmatrix}p_{1}'\\
\vdots\\
p_{n}'
\end{bmatrix}_{E^{t}(\Gamma')}.
\]

We can present $C(\G\sqcup\G')$ and $C(\G\cup_{\zz}\G')$ as the
following Koszul complexes:
\begin{equation}
C(\G\sqcup\G')=\begin{bmatrix}p_{1}\\
\vdots\\
p_{m}\\
p_{1'}\\
\vdots\\
p_{n}'
\end{bmatrix}_{E^{t}(\Gamma\sqcup\G')}\mbox{ and }C(\G\cup_{\zz}\G')=\begin{bmatrix}p_{1}\\
\vdots\\
p_{m}\\
p_{1'}\\
\vdots\\
p_{n}'
\end{bmatrix}_{E^{t}(\G\cup_{\zz}\G')}\label{eq:KoszulUnionandGlue}
\end{equation}
Here the distinction comes from the difference in total and edge rings.
$C(\G\sqcup\G')$ is a chain complex of free $E(\G\sqcup\G')$-modules
with the chain modules as direct sums of shifted copies of $E^{t}(\G\sqcup\G').$
However $C(\G\cup_{\zz}\G')$ is a chain complex of free $E(\G\cup_{\zz}\G')$-modules
with the chain modules as direct sums of shifted copies of $E^{t}(\G\cup_{\zz}\G').$
We now give another useful technique for simplifying the complexes
associated to marked MOY graphs, called \emph{mark removal.}
\begin{lem}
\label{lem:MarkRemoval} Suppose that $z$ is an internal variable
of a marked MOY graph $\G$ and $C(\G)$ is presented as a Koszul
complex of the sequence $\pp=p_{1},...,z-p_{i},...,p_{k}$ where $p_{1},...,p_{i},...,p_{k}\in E(\G)$.
Let $\psi:E^{t}(\Gamma)\to E^{t}(\Gamma)/(z-p_{i})$ be the quotient
map identifying $z$ with $p_{i}.$ Then, as complexes of $E(\G)$-modules,
\[
C(\G)\simeq\psi\left(\begin{bmatrix}p_{1}\\
\vdots\\
\widehat{z-p_{i}}\\
\vdots\\
p_{k}
\end{bmatrix}\right)\simeq\begin{bmatrix}p_{1}\\
\vdots\\
\widehat{z-p_{i}}\\
\vdots\\
p_{k}
\end{bmatrix}_{E^{t}(\G)/(z-p_{i})}
\]

omitting the term $z-p_{i}$ from the sequence.
\end{lem}
Various forms of this lemma are proven in other texts on HOMFLY-PT
homology, such as \cite{KR08b,Ras06}. We refer the reader to \cite{Ras06}
for this exact form. Lemma \ref{lem:MarkRemoval} allows us to freely
add in or remove marks without changing the homotopy type of the complex
(as a complex of $E(\G)$-modules). This implies the following very
helpful statement.
\begin{cor}
Let $\G$ and $\G'$ be two marked MOY graphs whose underlying (unmarked)
MOY graphs are the same (isomorphic as oriented graphs). Then $C(\Gamma)\simeq C(\Gamma')$
as complexes of $E(\G)=E(\G')$-modules.\end{cor}
\begin{example}
\label{exa:MOYKozsul}Consider the marked MOY graph from Figure \ref{fig:ExampleBuild}.
The marks partition the MOY graphs into six elementary MOY graphs
(three MOY vertices and three arcs) which are drawn below in Figure
\ref{fig:ExampleBuild}.

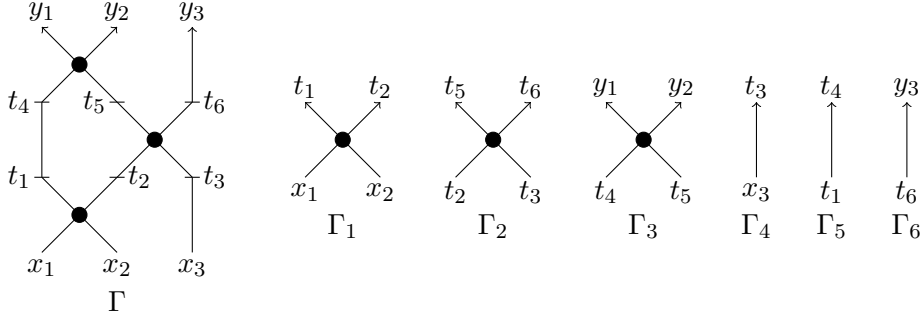
\begin{figure}[h]
\begin{tikzpicture} \draw (-3.5,-1)--(-2.5,0); \draw (-2.5,-1)--(-3.5,0); \draw (-1.5,-1)--(-1.5,0);\draw[fill=black] (-3,-0.5) circle (0.1cm); \draw (-3.5,0)--(-3.5,1); \draw (-2.5,0)--(-1.5,1); \draw (-1.5,0)--(-2.5,1);\draw[fill=black] (-2,0.5) circle (0.1cm); \draw[->] (-3.5,1)--(-2.5,2); \draw[->] (-2.5,1)--(-3.5,2); \draw[->] (-1.5,1)--(-1.5,2); \draw[fill=black] (-3,1.5) circle (0.1cm); \draw (-3.6,0)--(-3.4,0);\draw (-2.6,0)--(-2.4,0);\draw (-1.6,0)--(-1.4,0); \draw (-3.6,1)--(-3.4,1);\draw (-2.6,1)--(-2.4,1);\draw (-1.6,1)--(-1.4,1); \node at (-3.5,-1.2) {$x_1$};\node at (-2.5,-1.2) {$x_2$};\node at (-1.5,-1.2) {$x_3$}; \node at (-3.5,2.2) {$y_1$};\node at (-2.5,2.2) {$y_2$};\node at (-1.5,2.2) {$y_3$}; \node at (-3.8,0) {$t_1$};\node at (-2.2,0) {$t_2$};\node at (-1.2,0) {$t_3$}; \node at (-3.8,1) {$t_4$};\node at (-2.8,1) {$t_5$};\node at (-1.2,1) {$t_6$}; \node at (-2.5,-1.65) {$\Gamma$};
\draw[->] (0,0)--(1,1); \draw[->] (1,0)--(0,1); \draw[fill=black] (0.5,0.5) circle (0.1cm); \draw[->] (2,0)--(3,1); \draw[->] (3,0)--(2,1); \draw[fill=black] (2.5,0.5) circle (0.1cm); \draw[->] (4,0)--(5,1); \draw[->] (5,0)--(4,1); \draw[fill=black] (4.5,0.5) circle (0.1cm); \draw[->] (6,0)--(6,1); \draw[->] (7,0)--(7,1); \draw[->] (8,0)--(8,1);
\node at (0,-.2) {$x_1$};\node at (1,-.2) {$x_2$}; \node at (2,-.2) {$t_2$};\node at (3,-.2) {$t_3$}; \node at (4,-.2) {$t_4$};\node at (5,-.2) {$t_5$}; \node at (6,-.2) {$x_3$}; \node at (7,-.2) {$t_1$}; \node at (8,-.2) {$t_6$};
\node at (0,1.2) {$t_1$};\node at (1,1.2) {$t_2$}; \node at (2,1.2) {$t_5$};\node at (3,1.2) {$t_6$}; \node at (4,1.2) {$y_1$};\node at (5,1.2) {$y_2$}; \node at (6,1.2) {$t_3$}; \node at (7,1.2) {$t_4$}; \node at (8,1.2) {$y_3$};
\node at (0.5,-0.65) {$\G_1$}; \node at (2.5,-0.65) {$\G_2$}; \node at (4.5,-0.65) {$\G_3$}; \node at (6,-0.65) {$\G_4$};\node at (7,-0.65) {$\G_5$};\node at (8,-0.65) {$\G_6$}; \end{tikzpicture}

\caption{The marked MOY graph in Example \ref{exa:MOYKozsul} and its elementary
MOY graphs \label{fig:ExampleBuild}}

\end{figure}

We can write $\G$ as $(\G_{1}\sqcup\G_{4})\cup_{t_{1},t_{2},t_{3}}(\G_{2}\sqcup\G_{5})\cup_{t_{4},t_{5},t_{6}}(\G_{3}\sqcup\G_{6}),$
and therefore 
\[
C(\G)=C(\G_{1}\sqcup\G_{4})\otimes_{\QQ[t_{1},t_{2},t_{3}]}C(\G_{2}\sqcup\G_{5})\otimes_{\QQ[t_{4},t_{5},t_{6}]}C(\G_{3}\sqcup\G_{6})
\]

We can write $C(\G)$, after some applications of mark removal, as
\[
C(\G)\simeq\begin{bmatrix}y_{1}+y_{2}-t_{1}-t_{4}\\
y_{1}y_{2}-t_{1}t_{3}\\
t_{4}+y_{3}-t_{2}-x_{3}\\
t_{4}y_{3}-t_{2}x_{3}\\
t_{1}+t_{2}-x_{1}-x_{2}\\
t_{1}t_{2}-x_{1}x_{2}
\end{bmatrix}_{\QQ[x_{1},x_{2},x_{3},y_{1},y_{2},y_{3},t_{1},t_{2},t_{4}]}
\]

We invite the reader to finish the process of removing the internal
variables $t_{1},t_{2}\mbox{ and }t_{4}$ to get a \emph{finite rank}
complex of $\QQ[x_{1},x_{2},x_{3},y_{1},y_{2},y_{3}]$-modules.
\end{example}

\subsection{MOY braid graphs\label{sub:MOY-braid-graphs}}

A \emph{MOY braid graph }is a graph formed by taking a braid and replacing
every crossing with a MOY vertex, whose incoming and outgoing edges
are consistent with the orientation of the braid. The complexes associated
to MOY braid graphs and their ``braid closures'' satisfy the following
local relations (as proven in \cite{KR08b,Ras06}):
\begin{prop}
\label{prop:CatMOYMoves}Let $\G_{0},\G_{1a},\G_{1b},\G_{2a},\G_{2b},\G_{3a},\G_{3b},\G_{3c},\mbox{ and }\G_{3d}$
be MOY graphs as in Figure \ref{fig:PropCatMOY} below. Then

\begin{eqnarray}
C(\Gamma_{0}) & \simeq & \bigoplus_{i=0}^{\infty}q^{2i}(\QQ\oplus aq^{2}\QQ)\label{eq:CatMOY0}\\
C(\Gamma_{1a}) & \simeq & \bigoplus_{i=0}^{\infty}q^{2i}(C(\Gamma_{1b})\oplus aq^{4}C(\Gamma_{1b}))\label{eq:CatMOYI}\\
C(\Gamma_{2a}) & \simeq & C(\Gamma_{2b})\oplus q^{2}C(\Gamma_{2b})\label{eq:CatMOYII}\\
C(\Gamma_{3a})\oplus q^{2}C(\Gamma_{3b}) & \simeq & q^{2}C(\Gamma_{3c})\oplus C(\Gamma_{3d})\label{eq:CatMOYIII}
\end{eqnarray}

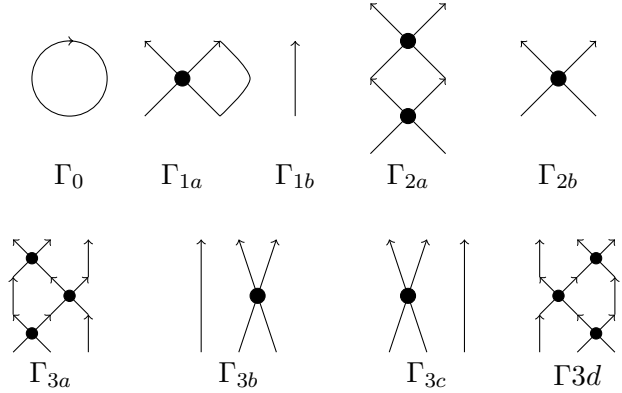
\begin{figure}[h]
\begin{tikzpicture}
\draw (0,1) circle (.5cm); \draw[->] (0,1.5)--(0.05,1.5);
\node at (0,-.3) {$\Gamma_0$};

\draw[->](1,0.5)--(2,1.5); \draw[->](2,0.5)--(1,1.5); \draw[fill=black](1.5,1) circle (0.1cm);
\draw plot [smooth] coordinates {(2,1.5)(2.4,1)(2,0.5)};
\node at (1.5,-.3) {$\Gamma_{1a}$};

\draw[->] (3,0.5)--(3,1.5);
\node at (3,-.3) {$\Gamma_{1b}$};

\draw[->](4,0)--(5,1); \draw[->](5,0)--(4,1); \draw[fill=black](4.5,0.5) circle (0.1cm);
\draw[->](4,1)--(5,2); \draw[->](5,1)--(4,2); \draw[fill=black](4.5,1.5) circle (0.1cm);
\node at (4.5,-.3) {$\Gamma_{2a}$};

\draw[->](6,0.5)--(7,1.5); \draw[->](7,0.5)--(6,1.5); \draw[fill=black](6.5,1) circle (0.1cm);
\node at (6.5,-.3) {$\Gamma_{2b}$};

\end{tikzpicture}

\vspace{.2in}

\begin{tikzpicture}
\draw[->](8,0)--(8.5,0.5);\draw[->](8.5,0)--(8,0.5);\draw[->] (9,0)--(9,0.5);
\draw[fill=black] (8.25,0.25) circle (0.075cm);
\draw[->](8,0.5)--(8,1);\draw[->](8.5,0.5)--(9,1);\draw[->] (9,0.5)--(8.5,1);
\draw[fill=black] (8.75,0.75) circle (0.075cm);
\draw[->](8,1)--(8.5,1.5);\draw[->](8.5,1)--(8,1.5);\draw[->] (9,1)--(9,1.5);
\draw[fill=black] (8.25,1.25) circle (0.075cm);
\node at (8.5,-.3) {$\Gamma_{3a}$};

\draw[->] (10.5,0)--(10.5,1.5); \draw[->](11,0)--(11.5,1.5); \draw[->](11.5,0)--(11,1.5);
\draw[fill=black] (11.25,0.75) circle (0.1cm);
\node at (11,-.3) {$\Gamma_{3b}$};

\draw[->] (13,0)--(13.5,1.5); \draw[->](13.5,0)--(13,1.5); \draw[->](14,0)--(14,1.5);
\draw[fill=black] (13.25,0.75) circle (0.1cm);
\node at (13.5,-.3) {$\Gamma_{3c}$};

\draw[->](15,0)--(15,0.5);\draw[->](15.5,0)--(16,0.5);\draw[->] (16,0)--(15.5,0.5);
\draw[fill=black] (15.75,0.25) circle (0.075cm);
\draw[->](15,0.5)--(15.5,1);\draw[->](15.5,0.5)--(15,1);\draw[->] (16,0.5)--(16,1);
\draw[fill=black] (15.25,0.75) circle (0.075cm);
\draw[->](15,1)--(15,1.5);\draw[->](15.5,1)--(16,1.5);\draw[->] (16,1)--(15.5,1.5);
\draw[fill=black] (15.75,1.25) circle (0.075cm);
\node at (15.5,-.3) {$\Gamma{3d}$};

\end{tikzpicture}

\caption{MOY graphs for Proposition \ref{prop:CatMOYMoves} \label{fig:PropCatMOY}}

\end{figure}

\end{prop}
To compare the isomorphisms in Proposition \ref{prop:CatMOYMoves}
to the relations in Figure \ref{fig:MOYRelations} we introduce the
notation of a ``Laurent series shift functor''. Suppose $F(q,a)\in\mathbb{N}\left\llbracket q^{\pm1},a^{\pm1}\right\rrbracket $,
that is that $F(q,a)=\sum_{i,j\in\ZZ}c_{ij}q^{i}a^{j}$. Also suppose
$M$ is a $\ZZ\times\ZZ$-graded $R$-module with grading shifts denoted
by $q^{i}a^{j}$. Then 
\begin{equation}
F(q,a)M:=\bigoplus_{i,j\in\ZZ}q^{i}a^{j}M^{\oplus c_{ij}}.\label{eq:PolyShiftFunctor}
\end{equation}

With this notation in mind, we can rewrite the isomorphisms (\ref{eq:CatMOY0})
and (\ref{eq:CatMOYI}). First looking at (\ref{eq:CatMOY0}): 
\begin{equation}
C(\Gamma_{0})\simeq\bigoplus_{i=0}^{\infty}q^{2i}(\QQ\oplus aq^{2}\QQ)=\bigoplus_{i=0}^{\infty}q^{2i}(1+aq^{2})\QQ=(1+aq^{2})\sum_{i=0}^{\infty}q^{2i}\QQ=\frac{1+aq^{2}}{1-q^{2}}\QQ.\label{eq:CatMOY0Rewrite}
\end{equation}

Next we consider the case of (\ref{eq:CatMOYI}):
\begin{equation}
C(\G_{1a})\simeq\bigoplus_{i=0}^{\infty}q^{2i}(C(\Gamma_{1b})\oplus aq^{4}C(\Gamma_{1b}))=(1+aq^{4})\sum_{i=0}^{\infty}q^{2i}C(\G_{1b})=\frac{1+aq^{4}}{1-q^{2}}C(\G_{1b}).\label{eq:CatMOYIRewrite}
\end{equation}

We invite the reader to compare the rewritten relations (\ref{eq:CatMOY0Rewrite})
and (\ref{eq:CatMOYIRewrite}) to (MOY 0) and (MOY I) from Figure
\ref{fig:MOYRelations}. The same comparison can be made for (\ref{eq:CatMOYII})
and (\ref{eq:CatMOYIII}) to (MOY 2a) and (MOY 3).

\subsection{Khovanov-Rozansky HOMFLY-PT homology\label{sub:KRHom}}

We now have all of the necessary tools to define Khovanov-Rozansky
HOMFLY-PT homology, or briefly HOMFLY-PT homology. We first define
two $q$-degree 0 maps $\chi_{i}:\arcs\to q^{-2}\wideedge$ and $\chi_{o}:\wideedge\to\arcs$.
Set $E=\QQ[x_{1},x_{2},y_{1},y_{2}]$ to be the edge ring of both
$\arcs$ and $\wideedge$. Then we define $\chi_{i}$ and $\chi_{o}$
by

\begin{equation}\label{eq:ChiInDef}
\begin{diagram}
\arcs & = & a^2q^4E & \rTo{\begin{psmallmatrix}y_1-x_1\\y_2-x_2\end{psmallmatrix}}
& aq^2E\oplus aq^2E & \rTo{\begin{psmallmatrix}x_2-y_2 & y_1-x_1\end{psmallmatrix}} & E\\
\dTo{\chi_i} & &\dTo{1} & & \dTo{\begin{psmallmatrix} y_1-x_2&0\\1&1 \end{psmallmatrix}} & & \dTo{y_1-x_2}\\
q^{-2}\wideedge &
= & a^2q^4E & \rTo_{\begin{psmallmatrix}(y_1-x_1)(y_1-x_2)\\y_1+y_2-x_1-x_2\end{psmallmatrix}} 
& aE \oplus aq^2E & \rTo_{\begin{psmallmatrix}x_1+x_2-y_1-y_2\\(y_1-x_1)(y_1-x_2)\end{psmallmatrix}} & q^{-2} E\\
\end{diagram}
\end{equation}

\begin{equation}\label{eq:ChiOutDef}
\begin{diagram}
\wideedge &
= & a^2q^6E & \rTo{\begin{psmallmatrix}(y_1-x_1)(y_1-x_2)\\y_1+y_2-x_1-x_2\end{psmallmatrix}} 
& aq^2E \oplus aq^4E & \rTo{\begin{psmallmatrix}x_1+x_2-y_1-y_2&(y_1-x_1)(y_1-x_2)\end{psmallmatrix}} &  E\\
\dTo{\chi_o} & &\dTo{y_1-x_2} & & \dTo{\begin{psmallmatrix} 1&0\\-1&y_1-x_2 \end{psmallmatrix}} & & \dTo{1}\\
\arcs & = & a^2q^4E & \rTo_{\begin{psmallmatrix}y_1-x_1\\y_2-x_2\end{psmallmatrix}}
& aq^2E\oplus aq^2E & \rTo_{\begin{psmallmatrix}x_2-y_2 & y_1-x_1\end{psmallmatrix}} & E\\
\end{diagram}
\end{equation}

We now define two bicomplexes of free $E$-modules for the positive
crossing $\posx$ and the negative crossing $\negx$: 

\begin{eqnarray}
C(\posx) & := & C(\arcs)\xrightarrow{\chi_{i}}tq^{-2}C(\wideedge)\label{eq:RouPos}\\
C(\negx) & := & t^{-1}C(\wideedge)\xrightarrow{\chi_{o}}C(\arcs)\label{eq:RouNeg}
\end{eqnarray}

Note that in the above diagram we use the notation $t^{k}C(\Gamma)$
to mean that the complex for $\Gamma$ sits in homological degree
$k$. This is a different homological degree than our Hochschild degree
we introduced earlier. We will simply denote this degree by $\deg_{t}$
and call it the \emph{homological degree. }We will say that $x$ has
(total) degree $\deg(x)=q^{i}a^{j}t^{k}$ if it has quantum degree
$i$, Hochschild degree $k$, and homological degree $j.$ We will
denote the differential in the complexes for the MOY graphs as $d_{g}$
and the differentials in the complexes (built from $\chi_{i}$ and
$\chi_{o}$) associated to crossings as $d_{c}$ . Both $C(\posx)$
and $C(\negx)$ are bicomplexes with commuting differentials $d_{g}$
and $d_{c}$. We define rings associated to each marked tangle diagram
$\tau$ in a similar manner to our constructions for marked MOY graphs.
We do not discuss this further but reference the reader to the table
in Figure \ref{fig:EdgeRingsTang}.

\begin{figure}[h]
\begin{tabular}{|c||c|c|c|}
\hline 
Notation & Name & Ring & Variables Included\tabularnewline
\hline 
\hline 
$E^{t}(\tau)$ & Total ring & $\mathbb{Q}[\mathbf{x},\mathbf{y},\mathbf{t}]$ & incoming, outgoing, internal\tabularnewline
\hline 
$E(\tau)$ & Edge ring & $\mathbb{Q}[\mathbf{x},\mathbf{y}]$ & incoming, outgoing\tabularnewline
\hline 
$E^{i}(\tau)$ & Incoming ring & $\mathbb{Q}[\mathbf{x}]$ & incoming\tabularnewline
\hline 
$E^{o}(\tau)$ & Outgoing ring & $\mathbb{Q}[\mathbf{y}]$ & outgoing\tabularnewline
\hline 
\end{tabular}\caption{Rings associated to a marked tangle diagram $\tau$\label{fig:EdgeRingsTang}}

\end{figure}

We can now build a bicomplex for any tangle diagram (and link diagram)
in a similar manner to what we did in $\S$\ref{sub:MOY-braid-graphs}
for MOY graphs. To a disjoint union of (marked) tangles $\tau=\tau_{1}\sqcup\tau_{2}$
we associate the bicomplex of $E(\tau)$-modules, 
\[
C(\tau):=C(\tau_{1})\otimes_{\QQ}C(\tau_{2})
\]

Where the tensor product of bicomplexes is defined by considering
a bicomplex as a complex of complexes and then applying the tensor
product of chain complexes. Similarly if we are gluing two tangles
$\tau_{1}$ and $\tau_{2}$ at the marked points $\zz=z_{1},...,z_{k}$
in such a way that the orientations are consistent, then we define
a bicomplex of $E(\tau_{1}\cup_{\zz}\tau_{2})$-modules 
\[
C(\tau_{1}\cup_{\zz}\tau_{2})=C(\tau_{1})\otimes_{\QQ[\zz]}C(\tau_{2})
\]

We omit the rest of the details in this case, and leave it to the
reader to compare with the analogous conventions for marked MOY graphs.
Now let $\beta\in\Br_{n}$ be a braid with $n$ strands. We can mark
$\beta$ in such a way that we partition it into arcs and crossings
of the form $\posx$ or $\negx$ and we label the endpoints and markings
in a similar manner to our conventions for marked MOY graphs. Therefore
we can use the rules of disjoint unions and gluing of tangles to write
a bicomplex $C(\beta)$ of $E(\beta)$-modules. In this case, $E(\beta)=\QQ[\xx,\yy]$
where $|\xx|=|\yy|=n.$

We now describe the construction of the HOMFLY-PT homology of a link
$L.$ Suppose that $\beta\in\Br_{n}$ is a braid representative of
$L,$ that is $L$ is the circular closure of $\beta$ in $\mathbb{R}^{3}$.
Then we can describe the bicomplex $C(L_{\beta})=C(\beta)\otimes_{\QQ[\xx,\yy]}C(1_{n})$,
where $1_{n}$ denotes the identity braid (oriented downwards) with
the top endpoints labeled by $\xx$ and the bottom endpoints labeled
by $\yy$. Here we use the notation $L_{\beta}$ to remember the choice
of braid representative. We refer the reader to Figure \ref{fig:braidClosure}
for an example of this decomposition of a braid closure.

\begin{figure}[h]
\begin{tikzpicture}
\draw (1,0)--(0,1); \draw (0,0)--(0.4,0.4); \draw (0.6,0.6)--(1,1);
\draw (1,1)--(0,2); \draw (0,1)--(0.4,1.4); \draw (0.6,1.6)--(1,2);
\draw[->] (1,2)--(0,3); \draw (0,2)--(0.4,2.4); \draw[->] (0.6,2.6)--(1,3);
\draw plot [smooth] coordinates {(0,3)(0.75,3.5)(1.75,3.5)(2.25,1.75)(1.75,-.5)(0.75,-.5)(0,0)};
\draw plot [smooth] coordinates {(1,3)(1.5,3.25)(1.75,1.5)(1.5,-0.25)(1,0)};
\node at (0.75,-1) {$L_\beta$};
\draw (-.1,0)--(.1,0);\draw (0.9,0)--(1.1,0);
\draw (-.1,3)--(.1,3);\draw (0.9,3)--(1.1,3);
\node at (-.2,.2) {$x_1$};\node at (1.2,.2) {$x_2$};
\node at (-.2,2.8) {$y_1$};\node at (1.2,2.8) {$y_2$};

\draw (5,0)--(4,1); \draw (4,0)--(4.4,0.4); \draw (4.6,0.6)--(5,1);
\draw (5,1)--(4,2); \draw (4,1)--(4.4,1.4); \draw (4.6,1.6)--(5,2);
\draw[->] (5,2)--(4,3); \draw (4,2)--(4.4,2.4); \draw[->] (4.6,2.6)--(5,3);
\node at (3.8,.2) {$x_1$};\node at (5.2,.2) {$x_2$};
\node at (3.8,2.8) {$y_1$};\node at (5.2,2.8) {$y_2$};
\node at (4.5,-1) {$\beta$};

\draw[->] (7.5,3)--(7.5,0); \draw[->] (8.5,3)--(8.5,0);
\node at (7.25,.2) {$x_1$};\node at (8.75,.2) {$x_2$};
\node at (7.25,2.8) {$y_1$};\node at (8.75,2.8) {$y_2$};
\node at (8,-1) {$1_2$};
\end{tikzpicture}

\caption{A link presented as a braid closure and the constituent tangles.\label{fig:braidClosure}}
\end{figure}
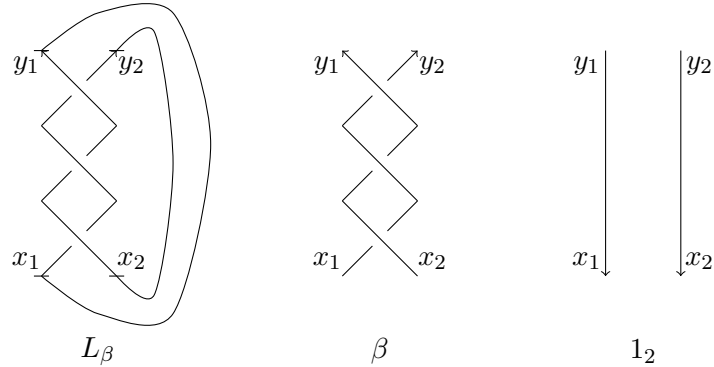

\begin{defn}
\label{def:KRHom}Suppose $L$ is a link with braid representative
$\beta\in\Br_{n}$. The \emph{HOMFLY-PT} homology of $L_{\beta}$
is $\HH(L_{\beta})=H_{d_{c*}}(H_{d_{g}}(C(L_{\beta})))$.\end{defn}
\begin{rem}
$\HH(L_{\beta})$, \emph{as defined above, arises as the $E^{2}$-page
of a spectral sequence. Let $L$ be an $n$-component link. It is
easily shown that the $E^{\infty}$-page of that spectral sequence
is the homology of the $n$-component unlink (up to a grading shift).
In particular, $H_{d_{c}}(C(L))$ is isomorphic to the $E^{\infty}$-page.}\end{rem}
\begin{thm}
[\cite{KR08b}]\label{thm:KRHom} Suppose $\beta\in\Br_{n}$ and
$\beta'\in\Br_{n'}$ are two braid representatives of a link $L$.
Then $\HH(L_{\beta})\cong\HH(L_{\beta'})$ up to a grading shift.
Furthermore, suppose the Poincaré series (see (\ref{eq:PoincareSeries}))
of $\HH(L_{\beta})$ is given by 
\[
\PP(L_{\beta})=\sum_{i,j,k\in\ZZ}d_{i,j,k}q^{i}a^{j}t^{k},\mbox{where }d_{i,j,k}=\dim_{\QQ}(\HH(L_{\beta}))_{i,j,k}
\]

then 
\[
\left.\PP(L_{\beta})\right|_{t=-1}=\sum_{i,j,k\in\ZZ}d_{i,j,k}q^{i}a^{j}(-1)^{k}=P(L_{\beta})
\]

\end{thm}

\section{HOMFLY-PT homology for general link diagrams\label{sec:GenDia}}

In this section we study what happens when we consider general link
diagrams in the construction of HOMFLY-PT homology. We will see that
not all Reidemiester moves are respected, and that in general HOMFLY-PT
homology is only an invariant up to braidlike isotopy.

\subsection{Virtual crossings and marked MOY graphs}

We start by introducing virtual crossings into the framework of (marked)
MOY graphs. We will not fully discuss virtual knot theory here, but
rather refer the reader to Kauffman \cite{Ka99}. Virtual crossings
were first considered as a tool in HOMFLY-PT and $\mathfrak{sl}$$(n)$
homologies by Khovanov and Rozansky in \cite{KR07}, and studied further
by the author and Rozansky in \cite{AbRoz14}.

A \emph{virtual MOY graph} is a MOY graph where we allow the underlying
graph to be nonplanar. Such a graph can always be drawn where the
intersections forced by the projection onto the plane are transverse
double points. An example of this is given in Figure \ref{fig:VirCrossing}.
To the marked virtual crossing graph we associate the following complex
of free $E(\vir)=\QQ[x_{1},x_{2},y_{1},y_{2}]$-modules:

\begin{equation}
C(\vir)=\begin{bmatrix}y_{1}-x_{2}\\
y_{2}-x_{1}
\end{bmatrix}_{E(\vir)}=q^{4}E(\vir)\xrightarrow{\,\,A\,\,}q^{2}E(\vir)\oplus q^{2}E(\vir)\xrightarrow{\,\,B\,\,}E(\vir)\label{eq:VirComplex}
\end{equation}
where 
\[
A=\begin{pmatrix}y_{1}-x_{2}\\
y_{2}-x_{1}
\end{pmatrix},\quad B=\begin{pmatrix}x_{1}-y_{2} & y_{1}-x_{2}\end{pmatrix}.
\]

Note that $C(\vir)$ resembles $C(\arcs)$ except for a transposition
of $x_{1}$ and $x_{2}$ in the definition of the complexes. In this
sense, we can think of a virtual crossing as being a permutation of
strands with no additional crossing data or vertex at the intersection.

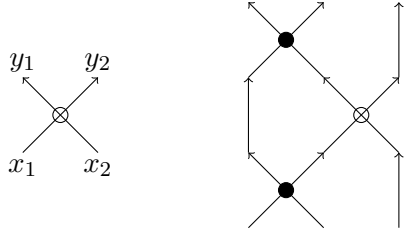
\begin{figure}[h]
\begin{tikzpicture}
\draw[->] (-3,1)--(-2,2); \draw[->] (-2,1)--(-3,2); \draw (-2.5,1.5) circle (0.1cm);
\node at (-3,0.8) {$x_1$};\node at (-2,0.8) {$x_2$};
\node at (-3,2.2) {$y_1$};\node at (-2,2.2) {$y_2$};

\draw[->] (0,0)--(1,1); \draw[->] (1,0)--(0,1); \draw[->] (2,0)--(2,1);
\draw[fill=black] (0.5,0.5) circle (0.1cm);
\draw[->] (0,1)--(0,2); \draw[->] (1,1)--(2,2); \draw[->] (2,1)--(1,2);
\draw (1.5,1.5) circle (0.1cm);
\draw[->] (0,2)--(1,3); \draw[->] (1,2)--(0,3); \draw[->] (2,2)--(2,3);
\draw[fill=black] (0.5,2.5) circle (0.1cm);

\end{tikzpicture}

\caption{A (marked) virtual crossing and an example of a virtual MOY graph
\label{fig:VirCrossing}}
\end{figure}

\begin{prop}
\label{prop:virtualMOYmoves} The moves in Figure \ref{fig:VirMOYMoves}
preserve the homotopy equivalence type of $C(\G)$.
\end{prop}
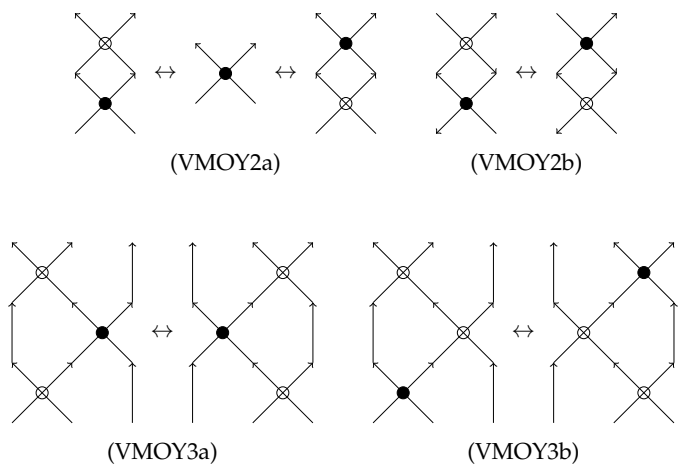
\begin{figure}[h]
\scalebox{.8}{ \begin{tikzpicture} \draw[->] (-2,0)--(-1,1); \draw[->] (-1,0)--(-2,1); \draw[fill=black] (-1.5,0.5) circle (0.1cm); \draw[->] (-2,1)--(-1,2); \draw[->] (-1,1)--(-2,2); \draw (-1.5,1.5) circle (0.1cm); \node at (-0.5,1) {$\leftrightarrow$}; \draw[->](0,.5)--(1,1.5); \draw[->](1,.5)--(0,1.5); \draw[fill=black](.5,1) circle (0.1cm); \node at (1.5,1) {$\leftrightarrow$}; \draw[->] (2,0)--(3,1); \draw[->] (3,0)--(2,1); \draw (2.5,0.5) circle (0.1cm); \draw[->] (2,1)--(3,2); \draw[->] (3,1)--(2,2); \draw[fill=black] (2.5,1.5) circle (0.1cm); \node at (0.5,-0.5) {(VMOY2a)};
\draw[<-] (4,0)--(5,1); \draw[->] (5,0)--(4,1); \draw[fill=black] (4.5,0.5) circle (0.1cm); \draw[->] (4,1)--(5,2); \draw[<-] (5,1)--(4,2); \draw (4.5,1.5) circle (0.1cm); \node at (5.5,1) {$\leftrightarrow$}; \draw[<-] (6,0)--(7,1); \draw[->] (7,0)--(6,1); \draw (6.5,0.5) circle (0.1cm); \draw[->] (6,1)--(7,2); \draw[<-] (7,1)--(6,2); \draw[fill=black] (6.5,1.5) circle (0.1cm); \node at (5.5,-0.5) {(VMOY2b)}; \end{tikzpicture}}

\vspace{0.3in}

\scalebox{.8}{\begin{tikzpicture}

\draw[->] (6,0)--(7,1);\draw[->] (7,0)--(6,1);\draw[->] (8,0)--(8,1);
\draw(6.5,0.5) circle (0.1cm);
\draw [->](6,1)--(6,2); \draw[->] (7,1)--(8,2); \draw[->] (8,1)--(7,2);
\draw[fill=black](7.5,1.5) circle (0.1cm);
\draw[->] (6,2)--(7,3); \draw[->] (7,2)--(6,3); \draw[->] (8,2)--(8,3);
\draw (6.5,2.5) circle (0.1cm);
\draw (8.5,1.5) node {$\leftrightarrow$};
\draw[->] (9,0)--(9,1);\draw[->](10,0)--(11,1);\draw[->](11,0)--(10,1);
\draw(10.5,0.5) circle (0.1cm);
\draw[->] (9,1)--(10,2); \draw[->] (10,1)--(9,2); \draw [->](11,1)--(11,2);
\draw[fill=black](9.5,1.5) circle (0.1cm);
\draw[->] (9,2)--(9,3); \draw[->] (10,2)--(11,3); \draw[->] (11,2)--(10,3);
\draw(10.5,2.5) circle (0.1cm);

\draw[->] (12,0)--(13,1);\draw [->](13,0)--(12,1);\draw[->] (14,0)--(14,1);
\draw[fill=black](12.5,0.5) circle (0.1cm);
\draw [->](12,1)--(12,2); \draw[->] (13,1)--(14,2); \draw[->] (14,1)--(13,2);
\draw(13.5,1.5) circle (0.1cm);
\draw[->] (12,2)--(13,3); \draw[->] (13,2)--(12,3); \draw[->] (14,2)--(14,3);
\draw(12.5,2.5) circle (0.1cm);
\draw (14.5,1.5) node {$\leftrightarrow$};
\draw [->](15,0)--(15,1);\draw [->](16,0)--(17,1);\draw [->](17,0)--(16,1);
\draw(16.5,0.5) circle (0.1cm);
\draw [->](15,1)--(16,2); \draw [->](16,1)--(15,2); \draw [->](17,1)--(17,2);
\draw(15.5,1.5) circle (0.1cm);
\draw[->] (15,2)--(15,3); \draw[->] (16,2)--(17,3); \draw[->] (17,2)--(16,3);
\draw[fill=black](16.5,2.5) circle (0.1cm);

\node at (8.5,-0.5) {(VMOY3a)};
\node at (14.5,-0.5) {(VMOY3b)};

\end{tikzpicture}}

\caption{Virtual MOY moves \label{fig:VirMOYMoves}}
\end{figure}

\begin{proof}
We shall not prove all of the isomorphisms due to the similarity of
their proofs. We shall prove (VMOY2a) and leave the rest of the proofs
to the reader. The left hand side of (VMOY2a) is presented as the
Koszul complex of the sequence $y_{2}-t_{1},y_{1}-t_{2},t_{1}+t_{2}-x_{1}-x_{2},(t_{1}-x_{1})(t_{1}-x_{2}).$
Let $w=y_{1}+y_{2}-x_{1}-x_{2}$, then 
\begin{equation}
\begin{aligned}\begin{bmatrix}y_{2}-t_{1}\\
y_{1}-t_{2}\\
t_{1}+t_{2}-x_{1}-x_{2}\\
(t_{1}-x_{1})(t_{1}-x_{2})
\end{bmatrix}_{\QQ[\xx,\yy,\mathbf{\mathbf{t}]}} & \simeq & \begin{bmatrix}y_{2}-t_{1}\\
y_{1}-t_{2}\\
w\\
(t_{1}-x_{1})(t_{1}-x_{2})
\end{bmatrix}_{\QQ[\xx,\yy,\mathbf{\mathbf{t}]}} & \simeq & \begin{bmatrix}y_{1}-t_{2}\\
w\\
(y_{1}-x_{1})(y_{1}-x_{2})
\end{bmatrix}_{\QQ[\xx,\yy,t_{2}]}\\
 &  &  & \simeq & \begin{bmatrix}w\\
(y_{1}-x_{1})(y_{1}-x_{2})
\end{bmatrix}_{\QQ[\xx,\yy]}.
\end{aligned}
\end{equation}

The first isomorphism is a change of basis and the other three isomorphisms
are mark removals. The last term is the Koszul complex for $C(\wideedge).$
The second isomorphism in (VMOY2a) is proven by a similar argument.
\end{proof}
For any \emph{virtual link diagram} $D$, that is a link diagram with
virtual crossings, we can repeat the procedure from Section \ref{sub:KRHom}
to build a bicomplex of $E(D)$-modules. We now record the additional
``virtual'' Reidemeister moves.
\begin{prop}
The moves in Figure \ref{fig:VirReidMoves} preserve the homotopy
equivalence type of $C(D)$. The isomorphisms (VR1),(VR2a),(VR2b),
(VR3) and (SVR) are called \emph{virtual Reidemeister moves,} and
the isomorphisms (Z1$\pm$) and (Z2$\pm$) are called \emph{Z-moves.}
\end{prop}
\begin{figure}[h]
\scalebox{.8}{\begin{tikzpicture}
\draw[->] (0,0.5)--(1,1.5); \draw[->] (1,0.5)--(0,1.5); \draw (0.5,1) circle (0.1cm);
\draw plot [smooth] coordinates {(1,1.5)(1.5,1)(1,0.5)};
\node at (1.75,1) {$\leftrightarrow$};
\draw[->] (2.25,0.5)--(2.25,1.5);
\node at (1.125,-.5) {(VR1)};

\draw (4,0)--(5,1); \draw (5,0)--(4,1); \draw (4.5,0.5) circle (0.1cm);
\draw[->] (4,1)--(5,2); \draw[->] (5,1)--(4,2); \draw (4.5,1.5) circle (0.1cm);
\node at (5.25,1) {$\leftrightarrow$};
\draw[->] (5.75,0.5)--(5.75,1.5); \draw[->] (6.75,0.5)--(6.75,1.5);
\node at (5.375,-.5) {(VR2a)};

\draw (8.5,0)--(9.5,1); \draw[<-] (9.5,0)--(8.5,1); \draw (9,0.5) circle (0.1cm);
\draw (8.5,1)--(9.5,2); \draw[->] (9.5,1)--(8.5,2); \draw (9,1.5) circle (0.1cm);
\node at (9.75,1) {$\leftrightarrow$};
\draw[->] (10.25,0.5)--(10.25,1.5); \draw[<-] (11.25,0.5)--(11.25,1.5);
\node at (9.875,-.5) {(VR2b)};
\end{tikzpicture}}

\vspace{0.3in}
\scalebox{.8}{
\begin{tikzpicture}
\draw (7.5,0)--(8.5,1);\draw (8.5,0)--(7.5,1);\draw (9.5,0)--(9.5,1);
\draw 	(8.1,0.5) arc (0:360:0.1);
\draw (7.5,1)--(7.5,2); \draw (8.5,1)--(9.5,2); \draw (9.5,1)--(8.5,2);
\draw 	(9.1,1.5) arc (0:360:0.1);
\draw[->] (7.5,2)--(8.5,3); \draw[->] (8.5,2)--(7.5,3); \draw[->] (9.5,2)--(9.5,3);
\draw 	(8.1,2.5) arc (0:360:0.1);
\draw (10,1.5) node {$\leftrightarrow$};
\draw (10.5,0)--(10.5,1); \draw (11.5,0)--(12.5,1); \draw(12.5,0)--(11.5,1);
\draw 	(12.1,0.5) arc (0:360:0.1);
\draw(10.5,1)--(11.5,2); \draw(11.5,1)--(10.5,2); \draw (12.5,1)--(12.5,2);
\draw 	(11.1,1.5) arc (0:360:0.1);
\draw[->] (10.5,2)--(10.5,3); \draw[->] (11.5,2)--(12.5,3); \draw[->] (12.5,2)--(11.5,3);
\draw 	(12.1,2.5) arc (0:360:0.1); 
\draw	 (10,-0.5) node {(VR3)};

\draw (14,0)--(15,1);\draw (15,0)--(14,1);\draw (16,0)--(16,1);
\draw 	(14.6,0.5) arc (0:360:0.1);
\draw (14,1)--(14,2); 
\draw (15,1)--(16,2); \draw (16,1)--(15.6,1.4); \draw (15.4,1.6)--(15,2);
\draw[->] (14,2)--(15,3); \draw[->] (15,2)--(14,3); \draw[->] (16,2)--(16,3);
\draw 	(14.6,2.5) arc (0:360:0.1);
\draw (16.5,1.5) node {$\leftrightarrow$};
\draw (17,0)--(17,1); \draw (18,0)--(19,1); \draw (19,0)--(18,1);
\draw 	(18.6,0.5) arc (0:360:0.1);
\draw (17,1)--(18,2); \draw (18,1)--(17.6,1.4); \draw (17.4,1.6)--(17,2);  
\draw (19,1)--(19,2);
\draw[->] (17,2)--(17,3); \draw[->] (18,2)--(19,3); \draw[->] (19,2)--(18,3);
\draw 	(18.6,2.5) arc (0:360:0.1); 
\draw	 (16.5,-0.5) node {(SVR)};
\end{tikzpicture}}

\vspace{.3in}

\scalebox{.8}{\begin{tikzpicture}
\draw (0,0)--(1,1); \draw (1,0)--(0,1); \draw (0.5,0.5) circle (0.1cm);
\draw[->] (0,1)--(1,2); \draw (1,1)--(0.6,1.4); \draw[->] (.4,1.6)--(0,2);
\node at (1.5,1) {$\leftrightarrow$};
\draw (2,0)--(3,1); \draw (3,0)--(2.6,0.4); \draw (2.4,0.6)--(2,1);
\draw[->](2,1)--(3,2); \draw[->] (3,1)--(2,2); \draw (2.5,1.5) circle (0.1cm); 

\draw (4,0)--(5,1); \draw (5,0)--(4,1); \draw (4.5,0.5) circle (0.1cm);
\draw[->] (5,1)--(4,2); \draw (4,1)--(4.4,1.4); \draw[->] (4.6,1.6)--(5,2);
\node at (5.5,1) {$\leftrightarrow$};
\draw (7,0)--(6,1); \draw (6,0)--(6.4,0.4); \draw (6.6,0.6)--(7,1);
\draw[->](6,1)--(7,2); \draw[->] (7,1)--(6,2); \draw (6.5,1.5) circle (0.1cm); 

\draw (8,0)--(9,1); \draw[<-] (9,0)--(8,1); \draw (8.5,0.5) circle (0.1cm);
\draw[->] (8,1)--(9,2); \draw (9,1)--(8.6,1.4); \draw (8.4,1.6)--(8,2);
\node at (9.5,1) {$\leftrightarrow$};
\draw (10,0)--(11,1); \draw[<-] (11,0)--(10.6,0.4); \draw (10.4,0.6)--(10,1);
\draw[->](10,1)--(11,2); \draw (11,1)--(10,2); \draw (10.5,1.5) circle (0.1cm); 

\draw (12,0)--(13,1); \draw[<-] (13,0)--(12,1); \draw (12.5,0.5) circle (0.1cm);
\draw[->] (13,1)--(12,2); \draw (12,1)--(12.4,1.4); \draw (12.6,1.6)--(13,2);
\node at (13.5,1) {$\leftrightarrow$};
\draw[<-]  (15,0)--(14,1); \draw (14,0)--(14.4,0.4); \draw (14.6,0.6)--(15,1);
\draw(14,1)--(15,2); \draw[->] (15,1)--(14,2); \draw (14.5,1.5) circle (0.1cm);

\node at (1.5,-0.5) {(Z1$+$)};
\node at (5.5,-0.5) {(Z1$-$)};
\node at (9.5,-0.5) {(Z2$+$)};
\node at (13.5,-0.5) {(Z2$-$)};
\end{tikzpicture}}

\caption{Virtual Reidemeister moves and Z-moves \label{fig:VirReidMoves}}
\end{figure}
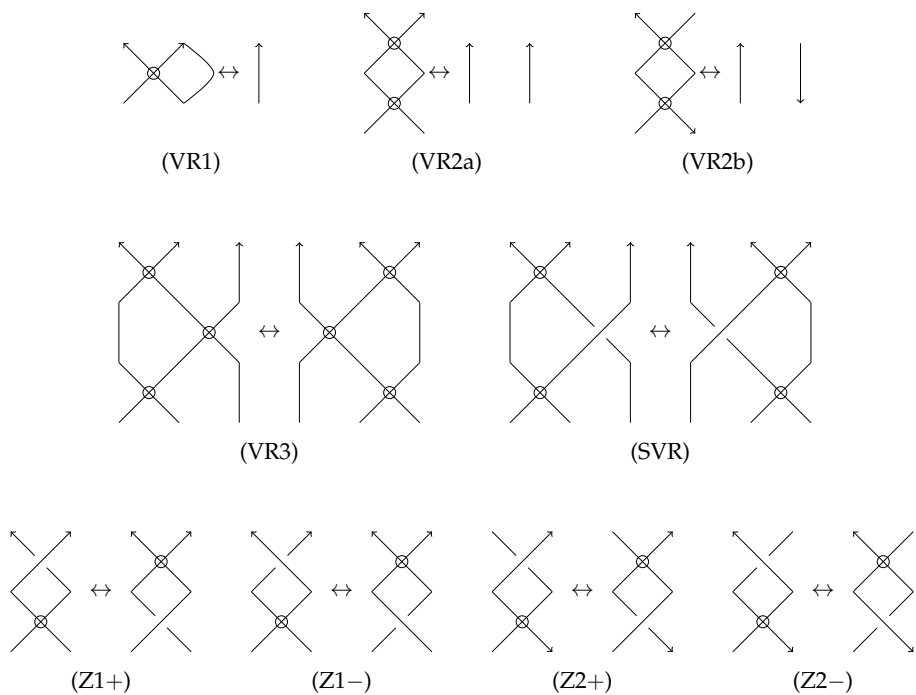

The proofs of (VR1), (VR2a), (VR2b), and (VR3) are similar to the
proof presented for (VMOY2a) in Proposition \ref{prop:virtualMOYmoves}.
The proofs for (SVR), and the Z-moves are included in \cite{AbRoz14}.
They are easy to see after resolving the crossings in terms of MOY
graphs and using the virtual MOY moves.

\subsection{Failure of Reidemeister IIb for $\HH(D)$ }

Now with the introduction of virtual crossings, we can study the exact
outcome of allowing general link diagrams in the computation of $\HH(D).$
We first look at the analogue of (MOYIIb) from Figure \ref{fig:MOYRelations}.
\begin{lem}
\label{lem:CatMOYIIb} 

Let $\tilde{C}(\arcsud)=\frac{q^{2}+aq^{6}}{1-q^{2}}C(\arcsud)$,
then the following diagram is commutative:

\begin{diagram}
\tilde{C}(\arcsud)&\rTo{\iota}&C(\edgeedge)&\rTo{\pi}&C(\viredge)\\
\dTo{\simeq}&&\dTo{\simeq}&&\dTo{\simeq}\\
\tilde{C}(\arcsud)&\rTo_{\begin{psmallmatrix}*\\1\end{psmallmatrix}}&C(\viredge)\oplus \tilde{C}(\arcsud)&\rTo_{\begin{psmallmatrix}1&0\end{psmallmatrix}}&C(\viredge)
\end{diagram}

where $\iota$ includes $\tilde{C}(\arcsud)$ as a subcomplex of $C(\edgeedge)$.
$\pi$ projects $C(\edgeedge)$ onto its quotient complex $C(\viredge)$
and $*$ represents some map from $\tilde{C}(\arcsud)$ to $C(\viredge)$.
\end{lem}
We prove Lemma \ref{lem:CatMOYIIb} in Section \ref{sub:AppenLemmaPf}
using the framework of virtual crossing filtrations from \cite{AbRoz14}.
We now state the main result of this section.
\begin{thm}
\label{thm:CatRIIb} The Reidemeister IIb isomorphism fails: $C(\posxnegxd)\not\simeq C(\arcsacross).$
In particular, $C(\posxnegxd)\simeq tq^{-2}C(\virnegxd).$\end{thm}
\begin{proof}
We first write $C(\posxnegxd)$ using (\ref{eq:RouPos}) and (\ref{eq:RouNeg}).
\[
C(\posxnegxd)=t^{-1}C(\arcedge)\xrightarrow{\begin{pmatrix}1\otimes\chi_{o}\\
\chi_{i}\otimes1
\end{pmatrix}}C(\arcarc)\oplus q^{-2}C(\edgeedge)\xrightarrow{\begin{pmatrix}\chi_{i}\otimes1 & -1\otimes\chi_{o}\end{pmatrix}}tq^{-2}C(\edgearc)
\]

We use the isomorphisms from Proposition \ref{prop:CatMOYMoves} and
Lemma \ref{lem:CatMOYIIb} to rewrite the complex as 
\[
t^{-1}q^{-2}\tilde{C}(\arcsud)\xrightarrow{\begin{pmatrix}1\\*
*\\*
\end{pmatrix}}q^{-2}\tilde{C}(\arcsud)\oplus C(\arcarc)\oplus q^{-2}C(\viredge)\xrightarrow{\begin{pmatrix}* & 1 & 0\\*
* & * & -1\otimes\chi_{o}
\end{pmatrix}}tC(\arcarc)\oplus tq^{-2}C(\arcsud)
\]
where the maps $*$ are irrelevant to our discussion. After performing
Gaussian elimination (see Section \ref{sub:TwistedCom}) on the complex
above, we are left with 

\begin{align*}
C(\posxnegxd) & \simeq q^{-2}C(\viredge)\xrightarrow{1\otimes\chi_{o}}tq^{-2}C(\arcsud)\\
 & \simeq tq^{-2}C(\vir)\otimes(t^{-1}C(\wideedged)\xrightarrow{\chi_{o}}C(\arcsd))\\
 & \simeq tq^{-2}C(\virnegxd).
\end{align*}

Where the tensor product above is over the appropriate ring of internal
variables.\end{proof}
\begin{example}
\label{exa:RIIbFail}We now give an example where the local failure
of Reidemeister IIb gives a failure of isotopy invariance for a certain
link diagram. Let $D$ be the diagram for the unknot given in Figure
\ref{fig:RIIbFail}, and let $O$ denote the standard diagram of the
unknot as a circle bounding a disc in $\mathbb{R}^{2}$. Then we have
the following chain of isomorphims 
\[
\HH(D)\simeq t^{2}q^{-4}\HH(D')\simeq t^{2}q^{-4}\HH(D'')\simeq t^{2}q^{-4}\HH(D''')
\]

The first isomorphism is given by applying Theorem \ref{thm:CatRIIb}
twice. The second isomorphism following from applying (Z2$-$) from
Proposition \ref{prop:virtualMOYmoves}. The last isomorphism follows
from applying (VR2b) from Proposition \ref{prop:virtualMOYmoves}.
$D'''$ is the diagram of the left-handed trefoil knot, and we computed
$P(D''')$ in Example \ref{exa:LHTrefoil}. This is enough to show
that $\HH(D)\not\simeq\HH(O)$, however it is an easy exercise to
show that 
\[
\HH(D''')=(aq^{2}+t^{2}q^{2}+t^{2}aq^{4})\frac{1+aq^{2}}{1-q^{2}}\QQ\not\simeq\frac{1+aq^{2}}{1-q^{2}}\QQ\simeq\HH(O).
\]

\end{example}
\begin{figure}[h]
\scalebox{.8}{\begin{tikzpicture}  \draw (1,0) node (v7) {}--(0,1); \draw (0,0) node (v11) {}--(0.4,0.4); \draw (0.6,0.6)--(1,1); \draw[->] (1,1)--(0,2) node (v9) {}; \draw (0,1)--(0.4,1.4); \draw[->] (0.6,1.6)--(1,2) node (v1) {}; \draw[->] (3.5,0.5) node (v5) {}--(2.5,1.5); \draw (2.5,0.5)--(2.9,0.9); \draw (3.1,1.1)--(3.5,1.5) node (v3) {}; \draw[<-] (4,0.5) node (v6) {}--(5,1.5) node (v10) {}; \draw[<-] (5,.5) node (v12) {}--(4.6,.9); \draw (4.4,1.1)--(4,1.5) node (v4) {}; \draw (1.5,0.5) node (v8) {}--(2.5,1.5); \draw[<-] (2.5,0.5)--(2.1,0.9); \draw (1.9,1.1)--(1.5,1.5) node (v2) {}; \draw  plot[smooth, tension=.7] coordinates {(v1) (1.3,1.8) (v2)}; \draw  plot[smooth, tension=.7] coordinates {(v3) (3.7,1.7) (v4)}; \draw  plot[smooth, tension=.7] coordinates {(v5) (3.7,0.3) (v6)}; \draw  plot[smooth, tension=.7] coordinates {(v7) (1.3,0.1) (v8)}; \draw  plot[smooth, tension=.7] coordinates {(v9) (1,2.5) (2.5,3) (4,2.5) (v10)}; \draw  plot[smooth, tension=.7] coordinates {(v11) (1,-0.5) (2.5,-1) (4,-0.5) (v12)}; \node at (2.5,-1.5) {$D$};
\draw (7,0) node (v7) {}--(6,1); \draw (6,0) node (v11) {}--(6.4,0.4); \draw (6.6,0.6)--(7,1); \draw[->] (7,1)--(6,2) node (v9) {}; \draw (6,1)--(6.4,1.4); \draw[->] (6.6,1.6)--(7,2) node (v1) {}; \draw[->] (9.5,0.5) node (v5) {}--(8.5,1.5); \draw (8.5,0.5)--(8.9,0.9); \draw (9.1,1.1)--(9.5,1.5) node (v3) {}; \draw[<-] (10,0.5) node (v6) {}--(11,1.5) node (v10) {}; \draw[<-] (11,0.5) node (v12) {}--(10.5,1); \draw (10.5,1)--(10,1.5) node (v4) {}; \draw (7.5,0.5) node (v8) {}--(8.5,1.5); \draw[<-] (8.5,0.5)--(8,1); \draw (8,1)--(7.5,1.5) node (v2) {}; \draw  plot[smooth, tension=.7] coordinates {(v1) (7.3,1.8) (v2)}; \draw  plot[smooth, tension=.7] coordinates {(v3) (9.7,1.7) (v4)}; \draw  plot[smooth, tension=.7] coordinates {(v5) (9.7,0.3) (v6)}; \draw  plot[smooth, tension=.7] coordinates {(v7) (7.3,0.1) (v8)}; \draw  plot[smooth, tension=.7] coordinates {(v9) (7,2.5) (8.5,3) (10,2.5) (v10)}; \draw  plot[smooth, tension=.7] coordinates {(v11) (7,-0.5) (8.5,-1) (10,-0.5) (v12)}; \node at (8.5,-1.5) {$D'$}; \draw (8,1) circle (.1cm); \draw (10.5,1) circle (.1cm);
\draw (1,-6) node (v7) {}--(0,-5); \draw (0,-6) node (v11) {}--(0.4,-5.6); \draw (0.6,-5.4)--(1,-5); \draw[->] (1,-5)--(0,-4) node (v9) {}; \draw (0,-5)--(0.4,-4.6); \draw[->] (0.6,-4.4)--(1,-4) node (v1) {}; \draw[->] (3.5,-5.5) node (v5) {}--(2.5,-4.5); \draw (2.5,-5.5)--(3,-5); \draw (3,-5)--(3.5,-4.5) node (v3) {}; \draw[<-] (4,-5.5) node (v6) {}--(5,-4.5) node (v10) {}; \draw[<-] (5,-5.5) node (v12) {}--(4.5,-5); \draw (4.5,-5)--(4,-4.5) node (v4) {}; \draw (1.5,-5.5) node (v8) {}--(1.9,-5.1); \draw (2.1,-4.9)--(2.5,-4.5); \draw[<-] (2.5,-5.5)--(2,-5); \draw (2,-5)--(1.5,-4.5) node (v2) {}; \draw  plot[smooth, tension=.7] coordinates {(v1) (1.3,-4.2) (v2)}; \draw  plot[smooth, tension=.7] coordinates {(v3) (3.7,-4.3) (v4)}; \draw  plot[smooth, tension=.7] coordinates {(v5) (3.7,-5.7) (v6)}; \draw  plot[smooth, tension=.7] coordinates {(v7) (1.3,-5.9) (v8)}; \draw  plot[smooth, tension=.7] coordinates {(v9) (1,-3.5) (2.5,-3) (4,-3.5) (v10)}; \draw  plot[smooth, tension=.7] coordinates {(v11) (1,-6.5) (2.5,-7) (4,-6.5) (v12)}; \node at (2.5,-7.5) {$D''$}; \draw (3,-5) circle (.1cm); \draw (4.5,-5) circle (.1cm);
\draw (8.2,-5.9) node (v7) {}--(7.2,-4.9);  \draw (7.2,-5.9) node (v11) {}--(7.6,-5.5); \draw (7.8,-5.3)--(8.2,-4.9); \draw[->] (8.2,-4.9)--(7.2,-3.9) node (v9) {}; \draw (7.2,-4.9)--(7.6,-4.5);  \draw[->] (7.8,-4.3)--(8.2,-3.9) node (v1) {}; \draw (8.7,-5.4) node (v8) {}--(9.1,-5); \draw (9.3,-4.8)--(9.7,-4.4) node (v13) {}; \draw[<-] (9.7,-5.4) node (v14) {}--(9.2,-4.9);  \draw (9.2,-4.9)--(8.7,-4.4) node (v2) {}; \node at (8.5,-7.3) {$D'''$}; \draw  plot[smooth, tension=.7] coordinates {(v1) (8.5,-3.7) (v2)}; \draw  plot[smooth, tension=.7] coordinates {(v7) (8.5,-6.1) (v8)}; \draw  plot[smooth, tension=.7] coordinates {(v13) (8.5,-3) (v9)}; \draw  plot[smooth, tension=.7] coordinates {(v11) (8.5,-6.7) (v14)}; \end{tikzpicture}}

\caption{The failure of Reidemeister IIb for an unknot diagram. The above diagrams
all have the same HOMFLY-PT homology up to a grading shift. \label{fig:RIIbFail}}
\end{figure}
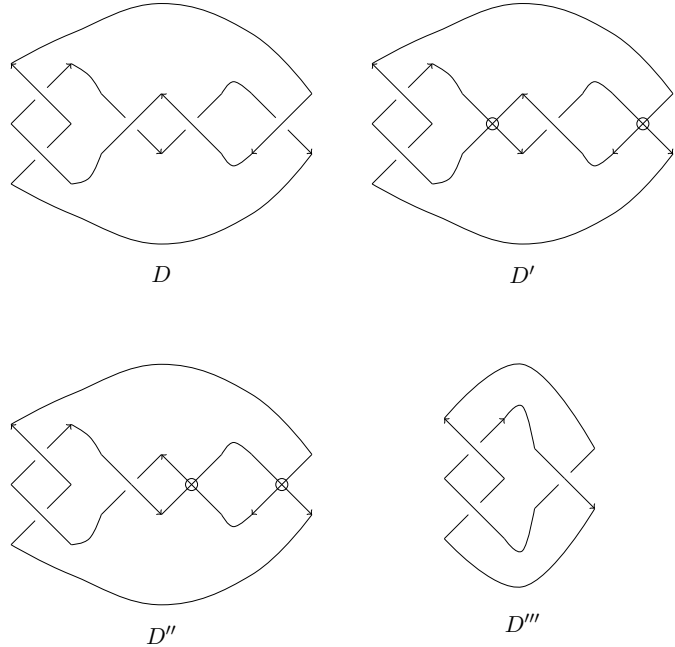

\subsection{Braidlike isotopy and $\HH(L)$}

As we saw in Example \ref{exa:RIIbFail}, HOMFLY-PT homology for general
link diagrams is \emph{not} an isotopy invariant. However, it was
proven in \cite{KR08b} that it was a \emph{braidlike isotopy }invariant.
We now carefully recall the definition of braidlike isotopy
\begin{defn}
\label{def:BraidlikeIso} Two link diagrams $D$ and $D'$ are said
to represent \emph{braidlike isotopic }links\emph{ }if they differ
by a sequence of planar isotopies and the following moves in Figure
\ref{fig:BraidlikeMoves}. Such a sequence of moves will be called
a \emph{braidlike isotopy. }

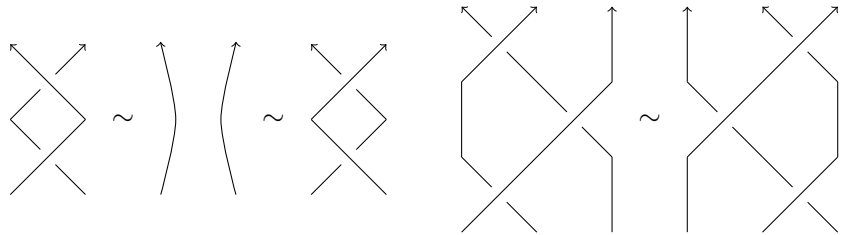
\begin{figure}[h]
\begin{tikzpicture}
\draw(0,0.5)--(1,1.5);\draw(1,0.5)--(0.6,0.9); \draw(0.4,1.1)--(0,1.5);
\draw(0,1.5)--(0.4,1.9);\draw[->](0.6,2.1)--(1,2.5); \draw[->](1,1.5)--(0,2.5);
\node at (1.5,1.5) {$\sim$};
\draw[->] plot [smooth] coordinates {(2,0.5)(2.2,1.5)(2,2.5)};
\draw[->] plot [smooth] coordinates {(3,0.5)(2.8,1.5)(3,2.5)};
\node at (3.5,1.5) {$\sim$};
\draw(4,0.5)--(4.4,0.9);\draw(4.6,1.1)--(5,1.5); \draw(5,0.5)--(4,1.5);
\draw[->](4,1.5)--(5,2.5);\draw(5,1.5)--(4.6,1.9); \draw[->](4.4,2.1)--(4,2.5);

\draw (6,0)--(7,1); \draw (7,0)--(6.6,0.4); \draw (6.4,0.6)--(6,1); \draw (8,0)--(8,1);
\draw (6,1)--(6,2); \draw (7,1)--(8,2); \draw (8,1)--(7.6,1.4); \draw(7.4,1.6)--(7,2);
\draw[->] (6,2)--(7,3); \draw (7,2)--(6.6,2.4); \draw[->] (6.4,2.6)--(6,3); \draw[->] (8,2)--(8,3);
\node at (8.5,1.5) {$\sim$};
\draw(9,0)--(9,1); \draw (10,0)--(11,1); \draw (11,0)--(10.6,0.4);\draw (10.4,0.6)--(10,1);
\draw(9,1)--(10,2); \draw (10,1)--(9.6,1.4); \draw (9.4,1.6)--(9,2); \draw (11,1)--(11,2);
\draw[->](9,2)--(9,3); \draw[->] (10,2)--(11,3); \draw (11,2)--(10.6,2.4);\draw[->] (10.4,2.6)--(10,3);
\end{tikzpicture}

\caption{Braidlike Reidemeister moves \label{fig:BraidlikeMoves}}
\end{figure}
\end{defn}
\begin{thm}
[Khovanov-Rozansky\cite{KR08b}]\label{thm:BriadlikeIso} Let $D$
and $D'$ be two braidlike isotopic link diagrams, then $\HH(D)\simeq\HH(D')$.
\end{thm}
Braidlike isotopy is an important notion in studying links in the
solid torus. It is a well-known fact that two braid closures in the
solid torus give isotopic links if and only if the braids are equivalent,
or rather if they are braidlike isotopic. Audoux and Fiedler in \cite{AF}
give a deformation of Khovanov homology which detects braidlike isotopy
of links in $\mathbb{R}^{3}$. Their invariant decategorifies to a
deformation of the Jones polynomial which can be computed using a
Kauffman bracket-like relation. In the case of closed braid diagrams,
their invariant corresponds with the homology theory studied in \cite{APS04}
and the decategorification corresponds with the polynomial invariant
studied in  \cite{HP89}. 

We can now interpret Example \ref{exa:RIIbFail} in the following
manner: The unknot diagram $D$ shown in Figure \ref{fig:RIIbFail}
is isotopic, but not braidlike isotopic to the standard unknot diagram
$O$. In particular, there is no sequence of Reidemeister moves transforming
$D$ to $O$ which does not contain the Reidemeister IIb move. In
this sense, we see that $\HH(D)$ can detect nonbraidlike isotopy.
Note that $\HH(D)$ is isomorphic to the homology of the left-handed
trefoil knot (after two negative Reidemeister I moves), though clearly
$D$ is not a diagram for the left-handed trefoil knot. Therefore,
this viewpoint of $\HH(D)$ may not be useful in determing isotopy
type of general diagrams, but can be very useful when we know the
two diagrams are of the same isotopy type and we wish to determine
if they are of the same braidlike isotopy type. 

We can also see easily that the Poincaré series of $\HH(D)$ and $\HH(O)$
differ. Direct calculation shows that $\PP(D)=(aq^{2}+t^{2}q^{2}+t^{2}aq^{4})\PP(O)$.
This implies that we may be able to detect nonbraidlike isotopies
on the level of the MOY calculus. As we will see in the next section,
after a deformation of the MOY theory, this is indeed the case.

\section{Decategorification of $\HH(D)$ for general link diagrams\label{sec:Decat}}

In this section we study the decategorification of $\HH(D),$ which
we will denote as $P^{b}(D).$ As we saw in Example \ref{exa:RIIbFail},
$P^{b}(D)\neq P(D)$ in general. In particular, when this occurs,
this implies that $D$ is not braidlike isotopic to a closed braid
presentation of a link. We will end this section with a note on virtual
links and give an explicit example of where $P^{b}(D)$ is not invariant
under the virtual exchange move, which implies that $\HH(D)$ cannot
be extended to a virtual link invariant.

\subsection{A deformation of the HOMFLY-PT polynomial}

Let $D$ be a link diagram, then we will define our deformed HOMFLY-PT
polynomial as 
\begin{equation}
P^{b}(D)=\left.\PP(D)\right|_{t=-1}\in\ZZ(q,a).\label{eq:DefPB}
\end{equation}

\begin{thm}
\label{thm:PBPoly} Let $D$ and $D'$ be two link diagrams which
are braidlike isotopic. Then $P^{b}(D)=P^{b}(D').$ $P^{b}$ satisfies
the following skein relation:
\begin{equation}
qP^{b}(\posx)-q^{-1}P^{b}(\negx)=(q-q^{-1})P^{b}(\arcs)\label{eq:BraidlikeSkein}
\end{equation}
\end{thm}
\begin{proof}
The first statement is an immediate corollary of Theorem \ref{thm:BriadlikeIso}.
For the second part of the statement, note that we have a map $\psi:tq^{-1}C(\negx)\to qC(\posx)$
of homological degree 1 given by

\begin{equation*}
\begin{diagram}
tq^{-1}C(\negx) =&q^{-1}C(\wideedge) & \rTo{\chi_o} & tq^{-1}C(\arcs)\\
\dTo{\psi} && \rdTo{1} &\\
qC(\posx)=&qC(\arcs) & \rTo{\chi_i}& tq^{-1}C(\wideedge).\\
\end{diagram}
\end{equation*}

The mapping cone of $\psi$, after Gaussian elimination, is homotopy
equivalent to 
\[
qC(\arcs)\xrightarrow{-\chi_{o}\chi_{i}}tq^{-1}C(\arcs)
\]

and therefore 
\begin{equation}
\mbox{Cone}(tq^{-1}C(\negx)\xrightarrow{\psi}C(\posx))\simeq(qC(\arcs)\xrightarrow{-\chi_{o}\chi_{i}}tq^{-1}C(\arcs))\label{eq:CatSkein}
\end{equation}

Properties of Poincaré series and (\ref{eq:CatSkein}) gives us the
relation (\ref{eq:BraidlikeSkein}) as desired.\end{proof}
\begin{cor}
If $D$ is a link diagram presented as a braid closure, then $P^{b}(D)=P(D).$
Equivalently if $D$ is a link diagram for some link $L$ and $P^{b}(D)\neq P(D)$,
then $D$ is not braidlike isotopic to a braid presentation of $L.$
\end{cor}
We can also give a MOY-style constuction for $P^{b}(D)$. In particular,
we address the relation (MOY IIb) from Figure \ref{fig:MOYRelations}.
As we saw in Lemma \ref{lem:CatMOYIIb}, the categorified MOY IIb
relation does not hold as we would expect. However, we can decategorify
the results in Proposition \ref{prop:CatMOYMoves}, Proposition \ref{prop:virtualMOYmoves},
Theorem \ref{thm:CatRIIb} and Lemma \ref{lem:CatMOYIIb} in a natural
way.
\begin{prop}
Let $D$ be a link diagram. The relations in Figure \ref{fig:MOYPb}
hold for $P^{b}(D).$
\end{prop}
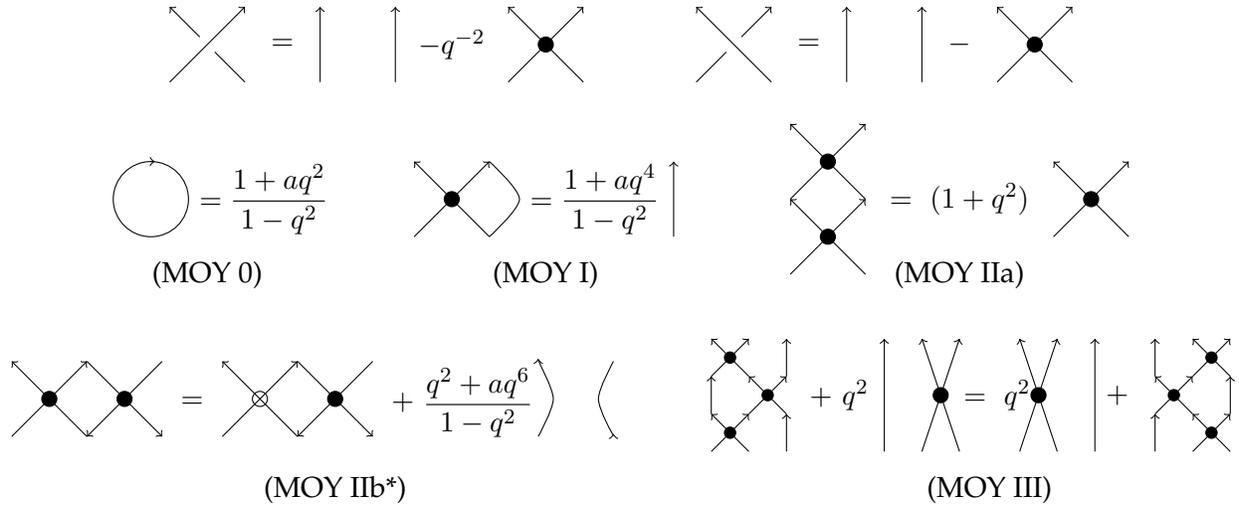
\begin{figure}[h]
\begin{tikzpicture} \draw[->] (3,0)--(4,1);\draw(4,0)--(3.6,0.4);\draw[->] (3.4,0.6)--(3,1); \node at (4.5,0.5) {$=$}; \draw[->] (5,0)--(5,1); \draw [->] (6,0)--(6,1);  \node at (6.75,0.5){$-q^{-2}$}; \draw[->] (7.5,0)--(8.5,1);\draw[->] (8.5,0)--(7.5,1);  \draw[fill=black] (8,0.5) circle (0.1cm);
\draw[->] (11,0)--(10,1);\draw(10,0)--(10.4,0.4);\draw[->] (10.6,0.6)--(11,1); \node at (11.5,0.5) {$=$}; \draw[->] (12,0)--(12,1); \draw [->] (13,0)--(13,1);  \node at (13.5,0.5){$-$}; \draw[->] (14,0)--(15,1);\draw[->] (15,0)--(14,1);  \draw[fill=black] (14.5,0.5) circle (0.1cm);
\end{tikzpicture}

\vspace{.2in}

\begin{tikzpicture} \draw (-3.5,1) circle (.5cm); \draw[->] (-3.5,1.5)--(-3.45,1.5); \node at (-2,1) {$=\dfrac{1+aq^2}{1-q^2}$}; \node at (-2.75,-0) {(MOY 0)};
\draw[->](0,0.5)--(1,1.5); \draw[->](1,0.5)--(0,1.5); \draw[fill=black](0.5,1) circle (0.1cm); \draw plot [smooth] coordinates {(1,1.5)(1.4,1)(1,0.5)}; \node at (2.4,1) {$= \dfrac{1+aq^4}{1-q^2}$}; \draw[->] (3.45,0.5)--(3.45,1.5);
\node at (1.75,-0) {(MOY I)}; \draw[->](5,0)--(6,1); \draw[->](6,0)--(5,1); \draw[fill=black](5.5,0.5) circle (0.1cm); \draw[->](5,1)--(6,2); \draw[->](6,1)--(5,2); \draw[fill=black](5.5,1.5) circle (0.1cm); \node at (7.25,1) {$=\,\,(1+q^2)$}; \draw[->](8.5,0.5)--(9.5,1.5); \draw[->](9.5,0.5)--(8.5,1.5); \draw[fill=black](9,1) circle (0.1cm);
\node at (7.25,-0) {(MOY IIa)};
\end{tikzpicture}

\vspace{.2in}

\begin{tikzpicture} \draw[->](-1.8,0.2)--(-0.8,1.2);\draw[->](-0.8,0.2)--(-1.8,1.2);\draw[fill=black] (-1.3,0.7) circle (0.1cm); \draw[->](-0.8,1.2)--(0.2,0.2);\draw[->](0.2,1.2)--(-0.8,0.2);\draw[fill=black] (-0.3,0.7) circle (0.1cm); \node at (0.6,0.65) {$=$}; \draw[->] (1,0.2)--(2,1.2); \draw[<-] (1,1.2)--(2,0.2); \draw (1.5,0.7) circle (0.1cm); \draw[<-] (2,0.2)--(3,1.2); \draw[->] (2,1.2)--(3,0.2); \draw[fill=black] (2.5,0.7) circle (0.1cm); \node at (3.4,0.65) {$+$}; \node at (4.4,0.65) {$\dfrac{q^2+aq^6}{1-q^2}$}; \draw[->] plot [smooth] coordinates {(5.2,0.2)(5.4,0.7)(5.2,1.2)}; \draw[<-] plot [smooth] coordinates {(6.2,0.2)(6,0.7)(6.2,1.2)}; \node at (2.5,-0.5) {(MOY IIb*)};

\draw[->](7.5,0)--(8,0.5);\draw[->](8,0)--(7.5,0.5);\draw[->](8.5,0)--(8.5,0.5); \draw[fill=black] (7.75,0.25) circle (0.075cm); \draw[->](7.5,0.5)--(7.5,1);\draw[->](8,0.5)--(8.5,1);\draw[->] (8.5,0.5)--(8,1); \draw[fill=black] (8.25,0.75) circle (0.075cm); \draw[->](7.5,1)--(8,1.5);\draw[->](8,1)--(7.5,1.5);\draw[->] (8.5,1)--(8.5,1.5); \draw[fill=black] (7.75,1.25) circle (0.075cm); \node at (9.2,0.75) {$+\,\, q^2$}; \draw[->] (9.8,0)--(9.8,1.5); \draw[->](10.3,0)--(10.8,1.5); \draw[->](10.8,0)--(10.3,1.5); \draw[fill=black] (10.55,0.75) circle (0.1cm); \node at (11.3,0.75) {$=\,\, q^2$}; \draw[->] (11.6,0)--(12.1,1.5); \draw[->](12.1,0)--(11.6,1.5); \draw[->](12.6,0)--(12.6,1.5); \draw[fill=black] (11.85,0.75) circle (0.1cm); \node at (12.9,0.75) {$+$}; \draw[->](13.4,0)--(13.4,0.5);\draw[->](13.9,0)--(14.4,0.5);\draw[->] (14.4,0)--(13.9,0.5); \draw[fill=black] (14.15,0.25) circle (0.075cm); \draw[->](13.4,0.5)--(13.9,1);\draw[->](13.9,0.5)--(13.4,1);\draw[->] (14.4,0.5)--(14.4,1); \draw[fill=black] (13.65,0.75) circle (0.075cm); \draw[->](13.4,1)--(13.4,1.5);\draw[->](13.9,1)--(14.4,1.5);\draw[->] (14.4,1)--(13.9,1.5); \draw[fill=black] (14.15,1.25) circle (0.075cm); \node at (11.2,-0.5) {(MOY III)}; \end{tikzpicture}

\caption{MOY relations for $P^{b}(D)$ and the deformed Reidemeister IIb relation.
The notation $P^{b}(\cdot)$ is omitted for readability.\label{fig:MOYPb}}
\end{figure}

\begin{example}
\label{exa:TorusKnotEx} Let $D$ be the diagram of the $(2,2k+1)$-torus
knot given in Figure \ref{fig:TorusKnotEx}. First note that we can
transform $D$ to $D'$ using (RIIb') so that $P^{b}(D)=q^{-4}P^{b}(D').$
$D'$ is isotopic to the $(2,2k-1)$-torus knot via a planar isotopy
and a (braidlike) Reidemeister IIa move. Therefore via a straight-forward
calculation in similar to Example \ref{exa:LHTrefoil},

\[
P^{b}(D)=\frac{1+aq^{2}}{1-q^{2}}\left((a+1)\sum_{i=1}^{k-1}q^{-i}+q^{-4k}\right).
\]

However, if $\tilde{D}$ is a braid presentation of the $(2,2k+1)$-torus
knot, then 
\[
P^{b}(\tilde{D})=\frac{1+aq^{2}}{1-q^{2}}\left((a+1)\sum_{i=1}^{k}q^{-i-2}+q^{-4k-2}\right).
\]

Therefore $D$ is not braidlike isotopic to a braid presentation of
the $(2,2k+1)$-torus knot for all $k\geq1$.

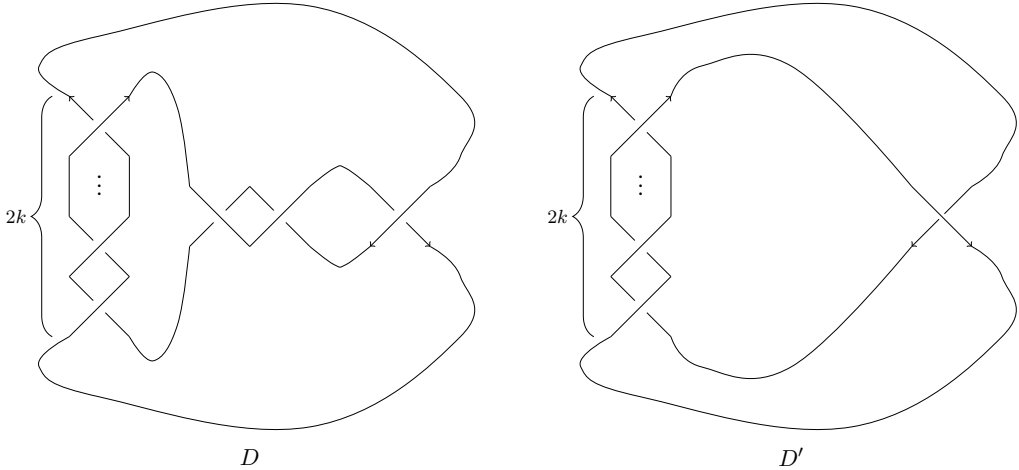
\begin{figure}[h]
\scalebox{0.8}{
\begin{tikzpicture} \usetikzlibrary{decorations.pathreplacing}
\draw (0,0) node (v12) {}--(1,1); \draw (1,0) node (v3) {}--(.6,.4); \draw (.4,.6)--(0,1); \draw (0,1)--(1,2); \draw (1,1)--(0.6,1.4); \draw (0.4,1.6)--(0,2); \draw (0,2)--(0,3); \draw (1,2)--(1,3); \node at (0.5,2.62) {$\vdots$}; \draw[->] (0,3)--(1,4) node (v1) {}; \draw (1,3)--(0.6,3.4); \draw[->] (0.4,3.6)--(0,4) node (v10) {}; \draw [decorate,decoration={brace,amplitude=10pt},xshift=-8pt,yshift=0pt] (0,0) -- (0,4.0) node [black,midway,xshift=-0.6cm]  {\footnotesize $2k$}; \draw[<-] (5,1.5) node (v8) {}--(6,2.5) node (v9) {}; \draw[<-] (6,1.5) node (v11) {}--(5.6,1.9); \draw (5.4,2.1)--(5,2.5) node (v6) {}; \draw (3,1.5)--(4,2.5) node (v5) {}; \draw (4,1.5) node (v7) {}--(3.6,1.9); \draw (3.4,2.1)--(3,2.5); \draw (3,1.5)--(2,2.5) node (v2) {}; \draw (2,1.5) node (v4) {}--(2.4,1.9); \draw (2.6,2.1)--(3,2.5); \draw  plot[smooth, tension=.7] coordinates {(v1) (1.4,4.4) (1.8,3.8) (v2)}; \draw  plot[smooth, tension=.7] coordinates {(v3) (1.4,-0.4) (1.8,0.2) (v4)}; \draw  plot[smooth, tension=.7] coordinates {(v5) (4.4,2.8) (4.6,2.8) (v6)}; \draw  plot[smooth, tension=.7] coordinates {(v7) (4.4,1.2) (4.6,1.2) (v8)}; \draw  plot[smooth, tension=.7] coordinates {(v9) (6.5,3) (6.5,4) (4,5.5) (0.5,5) (-0.5,4.5) (v10)}; \draw  plot[smooth, tension=.7] coordinates {(v11) (6.5,1) (6.5,0) (4,-1.5) (0.5,-1) (-0.5,-0.5) (v12)}; \node at (3,-2) {$D$};
\draw (9,0) node (v12) {}--(10,1); \draw (10,0) node (v3) {}--(9.6,0.4); \draw (9.4,0.6)--(9,1); \draw (9,1)--(10,2); \draw (10,1)--(9.6,1.4); \draw (9.4,1.6)--(9,2); \draw (9,2)--(9,3); \draw (10,2)--(10,3); \node at (9.5,2.62) {$\vdots$}; \draw[->] (9,3)--(10,4) node (v1) {}; \draw (10,3)--(9.6,3.4); \draw[->] (9.4,3.6)--(9,4) node (v10) {}; \draw [decorate,decoration={brace,amplitude=10pt},xshift=-8pt,yshift=0pt] (9,0) -- (9,4) node [black,midway,xshift=-0.6cm]  {\footnotesize $2k$}; \draw[<-] (14,1.5) node (v8) {}--(15,2.5) node (v9) {}; \draw(15,1.5) node (v11) {}--(14.6,1.9); \draw (14.4,2.1)--(14,2.5) node (v6) {}; \draw[ultra thick, white,->] (14,2.5)--(15,1.5); \draw[ultra thick, white,->] (14.05,2.5)--(15.05,1.5); \draw[ultra thick, white,->] (13.95,2.5)--(14.95,1.5); \draw[,->] (14,2.5)--(15,1.5);
\draw  plot[smooth, tension=.7] coordinates {(v7)}; \draw  plot[smooth, tension=.7] coordinates {(v9)}; \draw  plot[smooth, tension=.7] coordinates {(v9)}; \draw  plot[smooth, tension=.7] coordinates {(v9)}; \draw  plot[smooth, tension=.7] coordinates {(v9) (15.5,3) (15.5,4) (13,5.5) (9.5,5) (8.5,4.5) (v10)}; \draw  plot[smooth, tension=.7] coordinates {(v11) (15.5,1) (15.5,0) (13,-1.5) (9.5,-1) (8.5,-0.5) (v12)}; \node at (12,-2) {$D'$}; \draw  plot[smooth, tension=.7] coordinates {(v1) (10.5,4.5) (12,4.5) (v6)}; \draw  plot[smooth, tension=.7] coordinates {(v3) (10.5,-0.5) (12,-0.5) (v8)}; \end{tikzpicture}}

\caption{A nonbraidlike diagram $D$ for the $(2,2k+1)$-torus knot and another link diagram $D'$ such that $P^b(D)=q^{-4}P^b(D')$.
moves for $P^{b}(D)$\label{fig:TorusKnotEx}}
\end{figure}

\end{example}

\subsection{Obstruction to extending HOMFLY-PT homology to virtual links}

Finally we wish to present an argument showing that the current definition
of $\HH(D)$ cannot be extended to virtual links, even when they are
presented as closures of virtual braids. We will ultimately use $P^{b}(D)$
to justify this statement.

A \emph{virtual braid} is a braid in which we allow virtual crossings
alongside positive and negative crossings. Two virtual braids are
said to be equivalent if they differ by the moves from Figure \ref{fig:BraidlikeMoves}
and the moves (VR2a), (VR3), and (SVR) from Figure \ref{fig:VirReidMoves}.
Kauffman in \cite{Ka99} proves that every virtual link can be presented
as the closure of a virtual braid. There is also an analogue of the
Markov theorem for virtual links.
\begin{thm}
[Kamada \cite{Ka07}] Let $\beta$ and $\beta'$ be two virtual
braids. Their braid closures are equivalent virtual links if and only
if they differ by a sequence of virtual braid equivalence moves (the
braidlike Reidemeister moves, (VR2a), (VR3), and (SVR)), the Markov
moves, and the virtual exchange move. The Markov moves and virtual
exchange more are pictured in Figure \ref{fig:MarkovVir}.
\end{thm}
\begin{figure}[h]
\begin{tikzpicture}
\draw (0,0)--(0,0.6) node (v1) {}; \draw (1,0)--(1,0.6); \draw  (-0.1,0.6) rectangle (1.1,1.2); \draw (0,1.2)--(0,1.8) node (v1) {}; \draw (1,1.2)--(1,1.8); \draw  (-0.1,1.8) rectangle (1.1,2.4); \draw[->] (0,2.4)--(0,3); \draw[->] (1,2.4)--(1,3); \node at (0.5,1.5) {$\cdots$}; \node at (0.5,0.3) {$\cdots$}; \node at (0.5,2.7) {$\cdots$}; \node at (0.5,2.1) {$\beta$}; \node at (0.5,0.9) {$\alpha$};
\draw (2,0)--(2,0.6) node (v1) {}; \draw (3,0)--(3,0.6); \draw  (1.9,0.6) rectangle (3.1,1.2);
\draw (2,1.2)--(2,1.8) node (v1) {}; \draw (3,1.2)--(3,1.8); \draw  (1.9,1.8) rectangle (3.1,2.4);
\draw[->] (2,2.4)--(2,3); \draw[->] (3,2.4)--(3,3);
\node at (2.5,1.5) {$\cdots$}; \node at (6.5,0.3) {$\cdots$}; \node at (2.5,2.7) {$\cdots$}; \node at (2.5,2.1) {$\alpha$}; \node at (2.5,0.9) {$\beta$}; \node at (1.5,1.5) {$\leftrightarrow$}; \node at (1.5,-0.5) {Markov move I};
\node at (4.3,0.3) {$\cdots$}; \node at (4.3,2.7) {$\cdots$}; \node at (4.3,1.5) {$\cdots$}; \draw[->] (5.2,2.4)--(5.2,3); \draw[->] (4.6,2.4)--(4.6,3); \draw[->] (4,2.4)--(4,3); \draw (5.2,1.2)--(5.2,1.8); \draw (4.6,1.2)--(4.6,1.8); \draw (4,1.2)--(4,1.8); \draw (5.2,0)--(5.2,0.6); \draw (4.6,0)--(4.6,0.6); \draw (4,0)--(4,0.6); \draw (5.2,0)--(5.2,1.2); \draw (4,1.6)--(4,2.4); \draw (4.6,1.8)--(5.2,2.4); \draw (4.8,2.2)--(4.6,2.4); \draw (5.2,1.8)--(5,2); \draw  (3.9,1.2) rectangle (4.7,0.6); \node at (4.3,0.9) {$\beta$};
\node at (8.7,0.3) {$\cdots$}; \node at (8.7,2.7) {$\cdots$}; \node at (8.7,1.5) {$\cdots$}; \draw[->] (9.6,2.4)--(9.6,3); \draw[->] (9,2.4)--(9,3); \draw[->] (8.4,2.4)--(8.4,3); \draw (9.6,1.2)--(9.6,1.8); \draw (9,1.2)--(9,1.8); \draw (8.4,1.2)--(8.4,1.8); \draw (9.6,0)--(9.6,0.6); \draw (9,0)--(9,0.6); \draw (8.4,0)--(8.4,0.6); \draw (9.6,0)--(9.6,1.2); \draw (8.4,1.6)--(8.4,2.4); \draw (9,1.8)--(9.2,2); \draw (9.2,2.2)--(9,2.4); \draw (9.6,1.8)--(9,2.4); \draw  (8.3,1.2) rectangle (9.1,0.6); \node at (8.7,0.9) {$\beta$};
\node at (6.5,2.7) {$\cdots$};
\node at (6.5,1.5) {$\cdots$};
\draw[->] (7.4,2.4)--(7.4,3); \draw[->] (6.8,2.4)--(6.8,3); \draw[->] (6.2,2.4)--(6.2,3);
\draw (7.4,1.2)--(7.4,1.8); \draw (6.8,1.2)--(6.8,1.8); \draw (6.2,1.2)--(6.2,1.8);
\draw (7.4,0)--(7.4,0.6); \draw (6.8,0)--(6.8,0.6); \draw (6.2,0)--(6.2,0.6);
\draw (7.4,0)--(7.4,1.2); \draw (6.2,1.6)--(6.2,2.4); \draw (6.8,1.8)--(6.8,2.4); \draw (6.8,1.8)--(6.8,2.4); \draw (7.4,1.8)--(7.4,2.4);
\draw (9.6,2.4)--(9.4,2.2);
\draw  (6.1,1.2) rectangle (6.9,0.6); \node at (6.5,0.9) {$\beta$}; \node at (5.7,1.5) {$\leftrightarrow$}; \node at (7.9,1.5) {$\leftrightarrow$}; \node at (6.8,-0.5) {Markov move II};
\draw (10.5,0)--(10.5,0.5); \draw (11,0)--(11,0.5); \draw (11.5,0.5)--(11.5,0); \draw (11,0.5)--(11.5,1); \draw (11.5,0.5)--(11,1); \draw (10.5,0.5)--(10.5,1) node (v1) {}; \draw (11.5,1.5)--(11.5,1); \draw (10.5,2) node (v2) {}--(10.5,1.5); \draw (11,1.5)--(11.5,2); \draw (11,2)--(11.5,1.5); \draw[->] (10.5,2.5)--(10.5,3); \draw[<-] (11,3)--(11,2.5); \draw [<-] (11.5,3)--(11.5,2.5); \draw (11.5,2.5)--(11.5,2); \draw  (10.4,1) rectangle (11.1,1.5); \draw  (10.4,2) rectangle (11.1,2.5); \draw  (11.25,0.75) ellipse (0.1 and 0.1); \draw  (11.25,1.75) ellipse (0.1 and 0.1); \node at (10.75,2.25) {$\alpha$}; \node at (10.75,1.25) {$\beta$};
\node at (12,1.5) {$\leftrightarrow$};
\draw (12.5,0)--(12.5,0.5); \draw (13,0)--(13,0.5); \draw (13.5,0.5)--(13.5,0); \draw (13,0.5)--(13.2,0.7); \draw (13.5,1)--(13.3,0.8); \draw (13.5,0.5)--(13,1); \draw (12.5,0.5)--(12.5,1) node (v1) {}; \draw (13.5,1.5)--(13.5,1); \draw (12.5,2) node (v2) {}--(12.5,1.5); \draw (13,1.5)--(13.5,2); \draw (13.3,1.7)--(13.5,1.5); \draw (13,2)--(13.2,1.8); \draw[->] (12.5,2.5)--(12.5,3); \draw[<-] (13,3)--(13,2.5); \draw [<-] (13.5,3)--(13.5,2.5); \draw (13.5,2.5)--(13.5,2); \draw  (12.4,1) rectangle (13.1,1.5); \draw  (12.4,2) rectangle (13.1,2.5);
\node at (12.75,2.25) {$\alpha$}; \node at (12.75,1.25) {$\beta$}; \node at (12,-0.5) {Virtual exchange move}; \node at (10.775,0.25) {$\cdots$}; \node at (10.775,2.75) {$\cdots$};
\node at (12.775,0.25) {$\cdots$}; \node at (12.775,2.75) {$\cdots$};
\end{tikzpicture}

\caption{Markov moves for virtual links and the virtual exchange move. $\alpha$
and $\beta$ are virtual braids. \label{fig:MarkovVir}}
\end{figure}
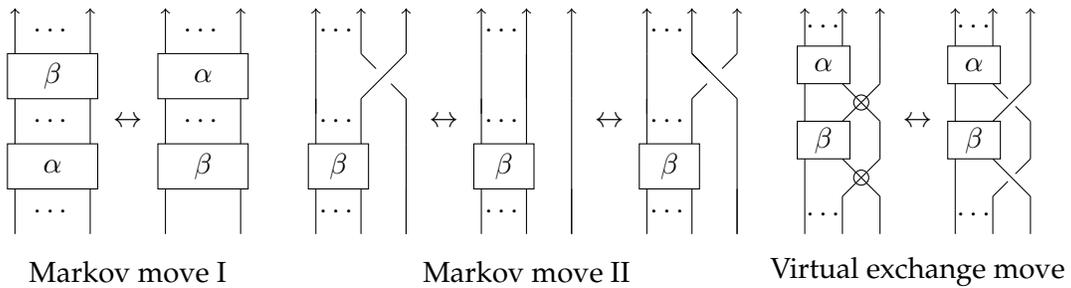

Now we show by example that $P^{b}(D)$ is \emph{not }invariant under
the virtual exchange move. Therefore $\HH(D)$ is not an invariant
of virtual links, even when the links are presented as virtual braid
closures.
\begin{example}
\label{exa:ExVirFail}Let $L$ be a connected sum of two virtual Hopf
links as shown in Figure \ref{fig:ExVirFail}. $\beta_{1}$ and $\beta_{2}$
are two virtual braids whose closures are equivalent as virtual links
to $L$. In particular, $\beta_{1}$ and $\beta_{2}$ are related
by Markov move I and the virtual exchange move shown in Figure \ref{fig:MarkovVir}.
Using the relations from Figure \ref{fig:MOYPb} we can directly compute
that

\begin{eqnarray*}
P^{b}(D_{1}) & = & \frac{1+aq^{2}}{1-q^{2}}\left(1-q^{-2}\left(\frac{1+aq^{4}}{1-q^{2}}\right)\right)^{2}\\
P^{b}(D_{2}) & = & \frac{1+aq^{2}}{1-q^{2}}\left(aq^{2}+2\left(\frac{1+aq^{4}}{1-q^{2}}\right)-q^{-2}\left(\frac{1+aq^{4}}{1-q^{2}}\right)^{2}\right)
\end{eqnarray*}

It is easy to see that $P^{b}(D_{2})\neq P^{b}(D_{1}).$ In particular,
\[
P^{b}(D_{2})-P^{b}(D_{1})=\dfrac{q^{2}(1+a+aq^{2}+a^{2}q^{2})}{1-q^{2}}
\]
 Therefore, $P^{b}(D)$ is not an invariant of virtual links and thus
neither is $\HH(D).$
\end{example}
\begin{figure}[H]
\begin{tikzpicture}                                               \draw[->] (-8.2,2) node (v4) {}--(-7.2,3); \draw (-7.2,2) node (v8) {}--(-7.6,2.4); \draw[->] (-7.8,2.6)--(-8.2,3); \draw (-8.2,3)--(-7.2,4) node (v7) {}; \draw (-7.2,3)--(-8.2,4) node (v2) {}; \draw (-7.7,3.5) circle (0.1cm); \draw[<-] (-10.2,2) node (v6) {}--(-9.2,3); \draw [<-](-9.2,2) node (v3) {}--(-9.6,2.4); \draw (-9.8,2.6)--(-10.2,3); \draw (-10.2,3)--(-9.2,4) node (v1) {}; \draw (-9.2,3)--(-10.2,4) node (v5) {}; \draw (-9.7,3.5) circle (0.1cm); \draw  plot[smooth, tension=.7] coordinates {(v1) (-8.9,4.3) (-8.7,4.5) (-8.5,4.3) (v2)}; \draw  plot[smooth, tension=.7] coordinates {(v3) (-8.9,1.7) (-8.7,1.5) (-8.5,1.7) (v4)}; \draw  plot[smooth, tension=.7] coordinates {(v5) (-10.6,3.6) (-11.2,3) (-10.6,2.4) (v6)}; \draw  plot[smooth, tension=.7] coordinates {(v7) (-6.8,3.6) (-6.2,3) (-6.8,2.4) (v8)}; \node at (-8.7,1) {$L$};   \end{tikzpicture}\hspace{1in} \scalebox{0.5}{\begin{tikzpicture} \draw (-4,0)--(-3,1); \draw (-3,0)--(-3.4,0.4); \draw (-3.6,0.6)--(-4,1); \draw (-4,1)--(-3,2); \draw (-3,1)--(-4,2); \draw (-3.5,1.5) circle (0.2cm); \draw (-4,3)--(-3,4); \draw (-3,3)--(-3.4,3.4); \draw (-3.6,3.6)--(-4,4); \draw (-4,4)--(-3,5); \draw(-3,4)--(-4,5); \draw (-3.5,4.5) circle (0.2cm); \draw (-2,0)--(-2,2); \draw (-4,2)--(-4,3); \draw (-3,2)--(-2,3); \draw (-2,2)--(-2.5,2.5); \draw (-2.5,2.5)--(-3,3); \draw (-2,3)--(-2,5); \draw[->] (-4,5)--(-4,6); \draw (-3,5)--(-2.5,5.5); \draw [->] (-2.5,5.5)--(-2,6); \draw[->] (-2,5)--(-3,6); \draw (-2.5,2.5) circle (0.2cm); \draw (-2.5,5.5) circle (0.2cm); \draw (0,0)--(1,1); \draw (1,0)--(.6,.4); \draw (.4,.6)--(0,1); \draw (0,1)--(1,2); \draw (1,1)--(0,2); \draw (.5,1.5) circle (0.2cm); \draw (0,3)--(1,4); \draw (1,3)--(0.6,3.4); \draw (0.4,3.6)--(0,4); \draw (0,4)--(1,5); \draw(1,4)--(0,5); \draw (0.5,4.5) circle (0.2cm); \draw (2,0)--(2,2); \draw (0,2)--(0,3); \draw (1,2)--(2,3); \draw (2,2)--(1.6,2.4); \draw (1.4,2.6)--(1,3); \draw (2,3)--(2,5); \draw[->] (0,5)--(0,6); \draw (1,5)--(1.4,5.4); \draw [->] (1.6,5.6)--(2,6); \draw[->] (2,5)--(1,6);                  \node at (-3,-0.75) {{\Huge$\beta_1$}}; \node at (1,-0.75) {\Huge$\beta_2$}; \end{tikzpicture}}

\caption{A connected sum of two virtual Hopf links, $L,$ and two braid presentations
of $L.$ \label{fig:ExVirFail}}
\end{figure}
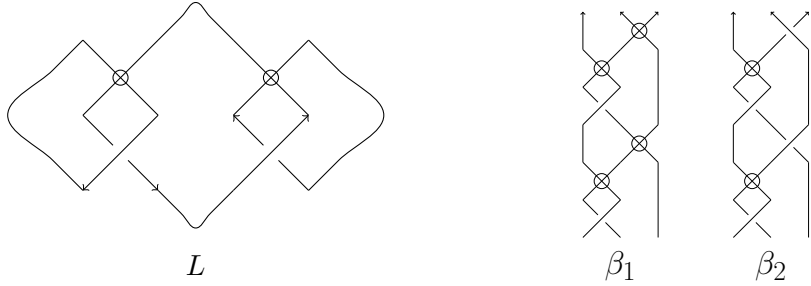

\appendix

\section{Virtual filtrations on HOMFLY-PT homology\label{sec:Virtual}}

In this appendix we present the necessary tools to understand our
application of virtual crossing filtrations from \cite{AbRoz14}.
We begin by reviewing relevant tools from homological algebra and
then introduce virtual crossing filtrations. We end the appendix by
presenting a proof of Lemma \ref{lem:CatMOYIIb}.

\subsection{Homological algebra of twisted complexes \label{sub:TwistedCom} }

For this section we assume that all chain complexes are chain complexes
of objects in some additive category. The main examples to keep in
mind are $R$-mod, the category of $R$-modules, and its homotopy
category $\mathcal{K}(R$-mod$)$.
\begin{defn}
\label{def:TwistedComplex}A \emph{twisted complex $C_{\bullet}$}
is a collection of chain complexes $C_{i}$ with maps $f_{ji}:C_{i}\to C_{j}$
for all $i,j\geq0$ such that \end{defn}
\begin{enumerate}
\item Only finitely many of the $C_{i}$ are nonzero.
\item $f_{ji}\in\bigoplus_{k\in\ZZ}\mbox{Hom}(C_{i,k+(j-i)},C_{j,k})$ if
$j>i$ and $f_{ji}=0$ otherwise.
\item $d_{C_{j}}f_{ji}+(-1)^{j-1+1}f_{ji}d_{C_{i}}=\sum_{j>k>i}f_{jk}f_{ki}$.
\end{enumerate}
We will say the \emph{length }of a map $f_{ji}$ in a twisted complex
is the number $j-i$. From condition (3) we can see that maps of length
one are chain maps, and that the composition of consecutative length
one maps $f_{j+1,j}f_{j,j-1}$ is homotopic to the zero map with homotopy
$f_{j+1,j-1}.$ Similar statements can be deduced for any length map.
\begin{defn}
Let $C_{\bullet}$ be a twisted complex such that $C_{k}=0$ for all
$k>n+1$, the \emph{convolution $\mbox{Con}(C_{\bullet})$ }is the
chain complex which is isomorphic to $\bigoplus_{i\geq0}C_{i}[i]$
as a module with differential 
\[
d_{\mbox{Con}(C_{\bullet})}=\begin{pmatrix}d_{C_{0}} & 0 & \cdots & \cdots & \cdots & 0\\
f_{10} & -d_{C_{1}} & 0 & \cdots & \cdots & 0\\
f_{20} & f_{21} & d_{C_{2}} & 0 & \cdots & 0\\
\vdots & \vdots &  & \ddots &  & \vdots\\
\vdots & \vdots &  &  & \ddots & 0\\
f_{n0} & f_{n1} & \cdots & f_{nk} & \cdots & (-1)^{n}d_{C_{n}}
\end{pmatrix}.
\]

\end{defn}
Note that $d_{Con(C_{\bullet})}^{2}=0$ by condition (3) of Definition
\ref{def:TwistedComplex}. We invite the reader to check that in the
case that $n=1$, $\mbox{Con}(C_{\bullet})\simeq\mbox{Cone}(f_{10})$,
and in the case that $n=2$ that $\mbox{Con}(C_{\bullet})$ is homotopy
equivalent to the iterated mapping cone 
\[
\mbox{ Cone}(\mbox{Cone}(f_{10})\xrightarrow{f_{20}+f_{21}}C_{2}[1])\simeq\Cone(C_{0}\xrightarrow{f_{20}+f_{10}}\Cone(f_{21}))
\]

We also recall a couple of facts which will be useful in simplifying
twisted complexes. Recall a \emph{deformation retract} $\Psi:C\to C'$
is a map $\Psi$ such that there exists a triple $(\Psi:C\to C',\Psi':C'\to C,\psi:C\to C)$
where $\Psi'$ is a chain map such that
\begin{enumerate}
\item $\Psi'\Psi=d_{C}\psi+\psi d_{C}+Id_{C}$
\item $\Psi\Psi'=Id_{C'}$
\item $\Psi\psi,\psi\Psi'$ and $\psi\psi$ are zero maps.
\end{enumerate}
We will often refer to the triple $(\Psi,\Psi',\psi)$ as \emph{SDR
data }for\emph{ $\Psi:C\to C'.$} We now recall a special type of
deformation retract.
\begin{prop}
\label{prop:GaussianElim}Consider the complex 

\begin{diagram}[labelstyle=\scriptstyle] A & \rTo^{\left(\begin{smallmatrix}\bullet \\ \alpha \end{smallmatrix}\right)} &  B \oplus C & \rTo{\left(\begin{smallmatrix}\varphi & \lambda \\ \mu & \nu\end{smallmatrix}\right)}& D \oplus E & \rTo{\left(\begin{smallmatrix} \bullet & \varepsilon  \end{smallmatrix}\right)}& F, \end{diagram}

where $\varphi:B\to D$ is an isomorphism and all other maps are arbitrary
up to the condition that $d^{2}=0.$ Then there exists a deformation
retract 

\begin{diagram}[labelstyle=\scriptstyle] A & \rTo{\left(\begin{smallmatrix}\bullet \\ \alpha \end{smallmatrix}\right)} &  B \oplus C & \rTo{\left(\begin{smallmatrix}\varphi & \lambda \\ \mu & \nu\end{smallmatrix}\right)}& D \oplus E & \rTo{\left(\begin{smallmatrix} \bullet & \varepsilon  \end{smallmatrix}\right)}& F \\ \dTo^{(1)} \uTo_{(1)} && \dTo^{\left(\begin{smallmatrix}0 & 1 \end{smallmatrix}\right)} \uTo_{\left(\begin{smallmatrix}-\varphi^{-1}\lambda \\ 1 \end{smallmatrix}\right)} && \dTo^{\left(\begin{smallmatrix}-\mu\varphi^{-1} & 1 \end{smallmatrix}\right)} \uTo_{\left(\begin{smallmatrix}0\\1 \end{smallmatrix}\right)} && \dTo^{(1)} \uTo_{(1)}\\ A & \rTo{\left(\begin{smallmatrix}\alpha \end{smallmatrix}\right)} &  C & \rTo{\left(\begin{smallmatrix}\nu - \mu\varphi^{-1}\varepsilon\end{smallmatrix}\right)}& E & \rTo{\left(\begin{smallmatrix}  \varepsilon  \end{smallmatrix}\right)}& F\\ \\ \end{diagram}

We call this deformation retract \emph{Gaussian elimination.}
\end{prop}
A proof of this propostion can be found in other texts on link homology
such as \cite{B-N05}. We now consider the effect of applying deformation
retracts to constituent complexes in a twisted complex. We write the
following fact in the case that the original twisted complex has only
maps of length 1, but the result can be extended to the general case.
\begin{prop}
\label{prop:TwistedCon}Let $C_{i}$ be chain complexes for $i=1,...,n$
and suppose that $(\Psi_{i},\Psi_{i}',\psi_{i})$ is SDR data for
the deformation retract $\Psi_{i}:C_{i}\to C_{i}'$. Suppose that
$C_{\bullet}=(C_{i},f_{ji})$ is a twisted complex such that $f_{ji}=0$
for $j-i\geq2$. Then there exist maps $f_{ji}':C_{i}'\to C_{j}'$
such that $C_{\bullet}'=(C_{i}',f_{ji}')$ is a twisted complex and
there exists a deformation retract $\nabla:Con(C_{\bullet})\to Con(C_{\bullet}').$
Explicitly, the maps $f_{ji}'$ are given by the formula
\[
f_{ji}'=\Psi_{j}f_{j,j-1}\psi_{j-1}f_{j-1,j-2}\psi_{j-2}\cdots f_{i+1,i}\Psi_{i}'.
\]

\end{prop}

\subsection{Virtual crossing filtrations\label{sub:AppenVirFil} }

Recall the complexes $C(\arcs)$, $C(\vir)$, and $C(\wideedge)$
given by 
\[
C(\arcs)=\left[\begin{array}{c}
y_{1}-x_{1}\\
y_{2}-x_{2}
\end{array}\right]_{E},\quad C(\vir)=\begin{bmatrix}y_{1}-x_{2}\\
y_{2}-x_{1}
\end{bmatrix}_{E},\quad C(\wideedge)=\begin{bmatrix}y_{1}+y_{2}-x_{1}-x_{2}\\
(y_{1}-x_{1})(y_{1}-x_{2})
\end{bmatrix}_{E},
\]

where $E=\QQ[x_{1},x_{2},y_{1},y_{2}].$ 
\begin{defn}
\label{def:GradedHom}Let $R$ be a commutative ring and let $C$
and $D$ be chain complexes of objects in an additive category. Let
$\bullet${[}i{]} denote the homological shift functor given by $C_{j+i}[i]=C_{j}$.
Define $\mbox{Hom}^{k}(C,D)$ to be the $\ZZ$-module of chain maps
$f:C\to D[-k]$ quotiented by the submodule of chain maps homotopic
to the zero map.

In \cite{KR07} Khovanov and Rozansky make the following observation.\end{defn}
\begin{prop}
There exists a unique map $F\in\mbox{Hom}^{1}(C(\arcs),q^{2}C(\vir))$,
up to rescaling, such that $\mbox{Cone}(F)$ is homotopy equivalent
to $C(\wideedge)$. Likewise there exists a unique map, up to rescaling,
$G\in\mbox{Hom}^{1}(C(\vir),q^{2}C(\arcs))$ such that $\mbox{Cone}(G)$
is homotopy equivalent to $C(\wideedge).$
\end{prop}
We will call the maps $F$ and $G$ \emph{virtual saddle maps.} In
our presentation of $C(\arcs)$ and $C(\vir)$ we can write the virtual
saddle maps explicitly. We write $G$ below and leave it to the reader
to write the analogous map $F$.

\begin{figure}[h]
\begin{equation*}
\begin{diagram}
a^2q^4E & \rTo{\left(\begin{smallmatrix} x_1-y_2\\x_2-y_1 \end{smallmatrix}\right)} & aq^2 E\oplus aq^2 E & \rTo{\left(\begin{smallmatrix} x_2-y_1 & y_2-x_1 \end{smallmatrix}\right)}& E\\
  & \rdTo{\left(\begin{smallmatrix} \,1 \\ -1 \end{smallmatrix} \right)}&            &\rdTo{\left(\begin{smallmatrix} 1 & 1 \end{smallmatrix} \right)} &  \\
a^2q^6E & \rTo_{\left(\begin{smallmatrix} x_1-y_1\\x_2-y_2 \end{smallmatrix}\right)} & aq^4 E \oplus aq^4 E & \rTo_{\left(\begin{smallmatrix} x_2-y_2 & y_1-x_1 \end{smallmatrix}\right)}& q^2E.\\
\end{diagram}
\end{equation*}

\caption{An explicit presentation of the virtual saddle map $G$. \label{fig:VirSaddle}}
\end{figure}

The mapping cone presentations give rise to filtrations. In particular,
$\mbox{Cone}(F)$ has $q^{2}C(\vir)$ as a submodule and $C(\arcs)$
as the quotient $Cone(F)/q^{2}C(\vir)$. We call this filtration the
\emph{negative filtration} and denote it as $C_{-}(\wideedge).$ Likewise
$\mbox{Cone}(G)$ has $q^{2}C(\arcs)$ as a submodule and $C(\vir)$
as the quotient $Cone(F)/q^{2}C(\arcs)$. We call this filtration
the \emph{positive filtration} and denote it as $C_{+}(\wideedge).$ 

We will often identify $C_{-}(\wideedge)$ with $\mbox{Cone}(F)$
(and $C_{+}(\wideedge)$ with $\mbox{Cone}(G)$) and consider the
filtered complexes as mapping cones. This process simplifies the differential
$d_{c}$ so that it can be presented in the following manner (as proven
in \cite{AbRoz14}). 
\begin{prop}
$C(\posx)$ is homotopy equivalent to the bicomplex $C(\arcs)\xrightarrow{\phi_{i}}tq^{-2}C_{+}(\wideedge)$,
where $\phi_{i}$ is the canonical inclusion of $C(\arcs)$ into $\mbox{Cone}(G).$
Suppose $C(\arcs)$ has the trivial filtration, then $\phi_{i}$ is
a filtered map with respect to the filtration on $C_{+}(\wideedge)$
and thus $C(\posx)$ is a filtered bicomplex. 

Likewise $C(\negx)$ is homotopy equivalent to the bicomplex $t^{-1}C(\wideedge)\xrightarrow{\phi_{o}}C(\arcs)$,
where $\phi_{i}$ is the canonical projection of $C(\arcs)$ from
$\mbox{Cone}(F).$ $\phi_{o}$ is a filtered map with respect to the
filtration on $C_{-}(\wideedge)$ and thus $C(\negx)$ is a filtered
bicomplex.
\end{prop}
We can extend this filtration to any tangle or link diagram via the
tensor product filtration. We will also refer to the given filtration
on the bicomplex associated to a tangle as a \emph{virtual crossing
filtration. }The following theorem was the main focus of \cite{AbRoz14}.
\begin{thm}
Let $\beta$ be a braid on $n$ strands, and $L_{\beta}$ denote its
circular closure. Then the virtual crossing filtration on $C(\beta)$
is invariant under Reidemeister IIa and is violated by at most two
levels by Reidemeister III. Furthermore, the virtual crossing filtration
on $\HH(L_{\beta})$ is invariant under the Markov moves (up to a
possible shift in filtration).
\end{thm}
We now focus on describing the filtrations on MOY graphs, which will
be useful in proving Lemma \ref{lem:CatMOYIIb}. A \emph{signed MOY
graph} is a MOY graph where each vertex is marked by a sign $+$ or
$-$. Suppose $\Gamma$ is a signed MOY graph marked so that it is
partitioned into graphs of the form $\wideedge$and $\arc$, then
we can define a filtration on $C(\Gamma).$ To each MOY vertex marked
with a $+$ we associate $C_{+}(\wideedge)$, and to each MOY vertex
marked with a $-$ we associate $C_{-}(\wideedge).$ We give the trivial
filtration to $C(\arc).$ Then the filtration on $C(\Gamma)$ is given
by the tensor product filtration. 

If we choose different sign assignments, then we receive homotopy
equivalent complexes for $C(\Gamma),$ but not necessarily \emph{filtered
}homotopy equivalent complexes (for example, $C_{-}(\wideedge)$ and
$C_{+}(\wideedge)$ are not filtered homotopy equivalent). For this
reason, if $\underline{\varepsilon}$ is a assignment of signs to
each MOY vertex of $\Gamma$, then we will write $C_{\underline{\varepsilon}}(\Gamma)$
for the filtered complex we get from the above construction. 

With this construction in mind, we may present every MOY graph as
an \emph{iterated mapping cone} or \emph{convolution} of graphs with
only virtual crossings and no MOY vertices. To do this we use twisted
complexes.
\begin{example}
We now consider the signed MOY graph $\Gamma=\edgeedge$ whose left
vertex is labeled by $+$ and right vertex is labeled by $-$. We
can present $C_{+-}(\Gamma)$ as the convolution of the following
twisted complex:

\begin{equation}
q^{4}C(\virarc)\xrightarrow{\begin{pmatrix}1\otimes F\\
G\otimes1
\end{pmatrix}}q^{2}C(\virvir)\oplus q^{2}C(\arcarc)\xrightarrow{\begin{pmatrix}-G\otimes1 & 1\otimes F\end{pmatrix}}C(\arcvir)\label{eq:MOYIIbCon}
\end{equation}

Note that this twisted complex does not have any nonzero maps of length
$\geq2$. 
\end{example}
We end this section with a couple of simplifications which will be
necessary in the next section. The following lemma is proven in \cite{AbRoz14}.
\begin{lem}
\label{lem:MultCoMult} Let $\hat{\vir}$ and $\hat{\arcs}$ denote
the partial braid closure of $\vir$and $\arcs$ respectively (see
Figure \ref{fig:PartialBraidClosure}). We write $\hat{G}:C(\hat{\vir})\to q^{2}C(\hat{\arcs})$
for the map induced by $G$ under braid closure, and similarily for
$\hat{F}.$ Then \end{lem}
\begin{enumerate}
\item $\Cone(\hat{G})\simeq\Cone(0)\simeq C(\arc)\oplus q^{2}C(\arc\circleO)$ 
\item $\Cone(\hat{F})\simeq\Cone\left(aq^{2}C(\arc)\oplus\dfrac{1+aq^{4}}{1-q^{2}}C(\arc)\xrightarrow{(1\,0)}q^{2}C(\arc)\right)$
\item $\Cone(\hat{F})\simeq\Cone(\hat{G})\simeq\dfrac{1+aq^{4}}{1-q^{2}}C(\arc)$
\end{enumerate}
\emph{Furthermore, the homotopy equivalences in (1) and (2) are filtered.}

\begin{figure}[h]

\scalebox{.8}{\begin{tikzpicture}
\draw (0,0)--(0,1); \draw (2,0)--(2,1); \draw (3,0)--(3,1); \node at (1,0.5) {$\cdots$};
\draw[->] (0,2)--(0,3); \draw[->] (2,2)--(2,3); \draw[->] (3,2)--(3,3); \node at (1,2.5) {$\cdots$};
\draw (-0.2,1) rectangle (3.2,2); \node at (1.5,1.5) {$\Gamma$};
\draw plot [smooth] coordinates {(3,3)(3.2,3.2)(3.4,2)(3.4,1)(3.2,-0.2)(3,0)};
\end{tikzpicture}}

\caption{Partial braid closure of a virtual braidlike MOY graph.\label{fig:PartialBraidClosure}}

\end{figure}
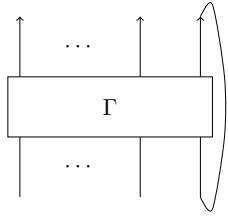

\subsection{A proof of Lemma \ref{lem:CatMOYIIb}\label{sub:AppenLemmaPf} }

We begin by restating Lemma \ref{lem:CatMOYIIb}.
\begin{lem}
[See Lemma \ref{lem:CatMOYIIb}]Let $\tilde{C}(\arcsud)=\frac{q^{2}+aq^{6}}{1-q^{2}}C(\arcsud)$,
then the following diagram is commutative:
\end{lem}
\begin{diagram}
\tilde{C}(\arcsud)&\rTo{\iota}&C(\edgeedge)&\rTo{\pi}&C(\viredge)\\
\dTo{\simeq}&&\dTo{\simeq}&&\dTo{\simeq}\\
\tilde{C}(\arcsud)&\rTo_{\begin{psmallmatrix}*\\1\end{psmallmatrix}}&C(\viredge)\oplus \tilde{C}(\arcsud)&\rTo_{\begin{psmallmatrix}1&0\end{psmallmatrix}}&C(\viredge)
\end{diagram}where $\iota$ includes $\tilde{C}(\arcsud)$ as a subcomplex of $C(\edgeedge)$.
$\pi$ projects $C(\edgeedge)$ onto its quotient complex $C(\viredge)$
and $*$ represents some map from $\tilde{C}(\arcsud)$ to $C(\viredge)$. 
\begin{proof}
We prove the center isomorphism as keeping track of the explicit maps
from application of Gaussian elimination constructs the maps on the
bottom row of the commutative diagram above. Consider $C(\edgeedge)$
with the virtual filtration given by marking the left vertex by $+$
and the right vertex by $-$, then $C(\edgeedge)$ can be written
as the convolution of the twisted complex given in (\ref{eq:MOYIIbCon}).
We can rewrite the twisted complex using the maps $\hat{F}$ and $\hat{G}$
from above, to get the twisted complex
\begin{equation*}
q^{4}C(\hat{\vir})\otimes_{\QQ}C(\arcd)\xrightarrow{\begin{pmatrix}1\otimes F\\
\hat{G}\otimes1
\end{pmatrix}}q^{2}C(\arcsacross)\oplus q^{2}C(\arcarc)\xrightarrow{\begin{pmatrix}-G\otimes1 & 1\otimes\hat{F}\end{pmatrix}}C(\arc)\otimes_{\QQ}C(\hat{\vird})
\end{equation*}

where $C(\arcarc)$ can be viewed as both $C(\hat{\arcs})\otimes_{\QQ}C(\arcd)$
and $C(\arc)\otimes_{\QQ}C(\hat{\arcsd})$. Lemma \ref{lem:MultCoMult}
states that $\hat{G}$ is homotopy equivalent to the zero map, so
we can simplify the above complex to 

\begin{center}
\begin{minipage}{175mm}
\includegraphics{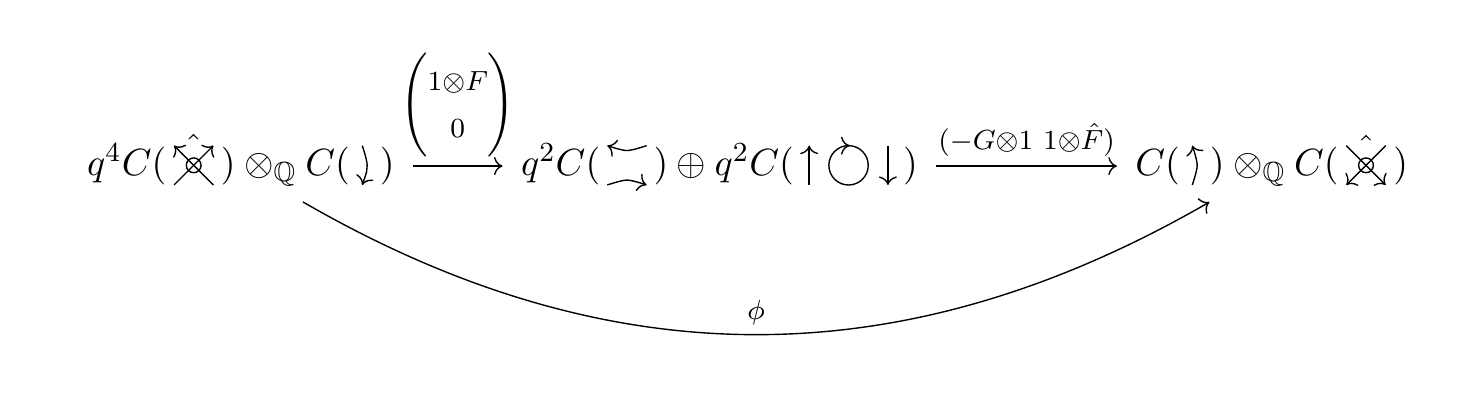}
\end{minipage}
\end{center}

Here the arrow $\phi$ represents a map of length two (as described
by Proposition \ref{prop:TwistedCon}) which has yet to be determined
explictly. We then use (2) of Lemma \ref{lem:MultCoMult} to rewrite
the above chain complex as
\begin{center}
\begin{minipage}{175mm}
\includegraphics{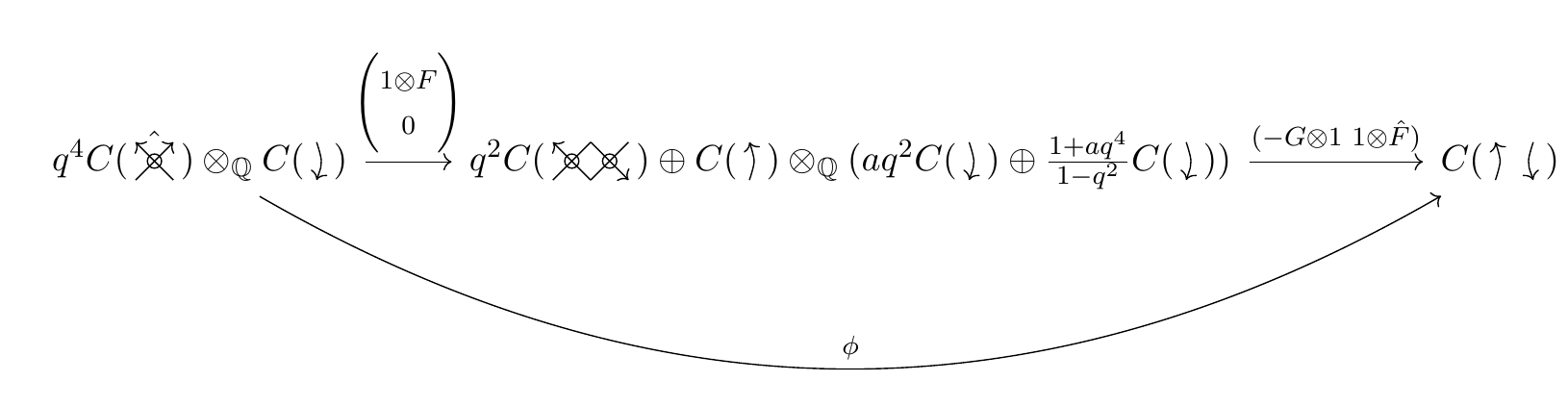}
\end{minipage}
\end{center}

Now we apply Proposition \ref{prop:GaussianElim} to the above twisted
complex to get the following twisted complex

\begin{center}
\begin{minipage}{122mm}
\includegraphics{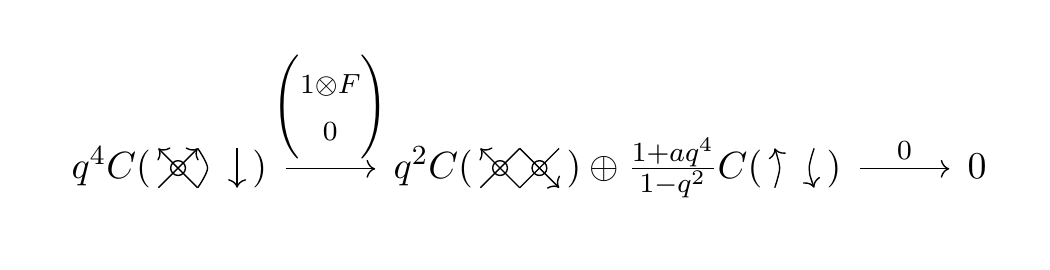}
\end{minipage}
\end{center}

Note that we now omit $\phi$ as it is clear it must be zero. The
convolution of the above twisted complex is homotopy equivalent to
$C(\viredge)\oplus\dfrac{q^{2}+aq^{6}}{1-q^{2}}C(\arcsud)$ as desired.
\end{proof}
\bibliographystyle{plain}
\bibliography{bib}

\end{document}